\documentclass[11pt]{amsart}
\usepackage{graphicx,amssymb,amsmath,amsthm}
\usepackage{enumerate}
\usepackage{dsfont}
\usepackage{booktabs}
\usepackage{mathrsfs}
\usepackage{verbatim}
\usepackage{epsfig}
\usepackage{lscape}
\usepackage{subfigure}
\usepackage[colorlinks, linkcolor=red, anchorcolor=blue, citecolor=green]{hyperref}
\usepackage{multirow}
\usepackage{epstopdf}
\usepackage{color}
\usepackage{caption}
\usepackage{subfigure}
\graphicspath{{Figures/}{logo/}}

\newcommand{\R}{{\mathbb R}}

\newcommand{\vu}{{\mathbf u}}
\newcommand{\vv}{{\mathbf v}}

\newcommand{\rank}{{\rm rank}}
\newcommand{\tr}{\operatorname{tr}}

\newcommand{\diag}{{\rm diag}}

\renewcommand{\omega}{\eta}

\newcommand{\RNum}[1]{\uppercase\expandafter{\romannumeral #1\relax}}

\theoremstyle{definition}
\newtheorem{definition}{Definition}[section]
\newtheorem{remark}[definition]{Remark}

\theoremstyle{plain}
\newtheorem{corollary}[definition]{Corollary}
\newtheorem{prop}[definition]{Proposition}
\newtheorem{theorem}[definition]{Theorem}
\newtheorem{lemma}[definition]{Lemma}

\newtheorem{problem}[definition]{Problem}
\newtheorem{conjecture}[definition]{Conjecture}

\numberwithin{equation}{section}

\date{}

\begin{document}
\baselineskip 14pt
\bibliographystyle{plain}

\title{Improved bounds in Weaver's ${\rm KS}_r$ conjecture for high rank positive semidefinite matrices}

\author{Zhiqiang Xu}
\thanks{}

\address{LSEC, Inst.~Comp.~Math., Academy of
Mathematics and System Science,  Chinese Academy of Sciences, Beijing, 100091, China
\newline
School of Mathematical Sciences, University of Chinese Academy of Sciences, Beijing 100049, China}
\email{xuzq@lsec.cc.ac.cn}

\author{Zili Xu}
\address{ LSEC, Inst.~Comp.~Math., Academy of
Mathematics and System Science,  Chinese Academy of Sciences, Beijing, 100091, China
\newline
School of Mathematical Sciences, University of Chinese Academy of Sciences, Beijing 100049, China}
\email{xuzili@lsec.cc.ac.cn}

\author{Ziheng Zhu}
\address{ LSEC, Inst.~Comp.~Math., Academy of
Mathematics and System Science,  Chinese Academy of Sciences, Beijing, 100091, China
\newline
School of Mathematical Sciences, University of Chinese Academy of Sciences, Beijing 100049, China}
\email{zhuziheng@lsec.cc.ac.cn}

\begin{abstract}
Recently Marcus, Spielman and Srivastava proved Weaver's ${\rm{KS}}_r$ conjecture, which gives a positive solution to the Kadison-Singer problem. In \cite{coh2016,Branden2}, Cohen and Br{\"a}nd{\'e}n independently extended this result to obtain the arbitrary-rank version of Weaver's ${\rm{KS}}_r$ conjecture.
In this paper, we present a new bound in Weaver's ${\rm{KS}}_r$ conjecture for the arbitrary-rank case. To do that, we introduce the definition of $(k,m)$-characteristic polynomials and employ it to improve the previous estimate on the largest root of the mixed characteristic polynomials. For the rank-one case, our bound agrees with the Bownik-Casazza-Marcus-Speegle's bound when $r=2$ \cite{bcms19}  and with the Ravichandran-Leake's bound when $r>2$ \cite{ravi2}. For the higher-rank case,  we sharpen the previous bounds from Cohen and from Br{\"a}nd{\'e}n.
\end{abstract}
\maketitle

\section{Introduction}

\subsection{The Kadison-Singer problem}

The Kadison-Singer problem, posed by Richard Kadison and Isadore Singer in 1959 \cite{KS}, is a fundamental problem that relates to a dozen areas of research in pure mathematics, applied mathematics and engineering. Basically, it asked whether each pure state on the diagonal subalgebra $l^{\infty}(\mathbb{N})$ to $\mathcal{B}(l^{2}(\mathbb{N}))$ has a unique extension. This problem was known to be equivalent to a large number of problems in analysis such as the Anderson Paving Conjecture \cite{And1,And2,And3}, Bourgain-Tzafriri Restricted Invertibility Conjecture \cite{BT,CT}, Feichtinger Conjecture \cite{BS06,CCLV05,Gro} and Weaver Conjecture \cite{Weaver}.

In a  seminal work \cite{inter2}, Marcus, Spielman and Srivastava resolved the Kadison-Singer problem by proving the Weaver's ${\rm{KS}}_r$ conjecture. The case $r=2$ of Weaver's ${\rm{KS}}_r$ conjecture can be stated as follows.

\begin{conjecture}{\rm {(${\rm{KS}}_2$)} }\label{KS2}
	There exist universal constants $\eta\geq 2$ and $\theta> 0$ such that the following holds. Let $\mathbf{u}_1,\ldots,\mathbf{u}_m\in\mathbb{C}^{d}$ satisfy $\|\mathbf{u}_i\|\leq1$ for all $i$ and
\begin{equation}
\sum\limits_{i=1}^{m}|\langle \mathbf{u},\mathbf{u}_i \rangle|^2=\eta
\end{equation}
for every unit vector $\mathbf{u}\in\mathbb{C}^{d}$. Then there exists a partition $S_1,S_2$ of $[m]:=\{1,\ldots,m\}$ such that
\begin{equation}
\sum\limits_{i\in S_j}^{}|\langle \mathbf{u},\mathbf{u}_i \rangle|^2\leq\eta-\theta
\end{equation}
for every unit vector $\mathbf{u}\in\mathbb{C}^{d}$ and each $j\in\{1, 2\}$.	
\end{conjecture}

The following theorem plays an important role in their proof of Weaver's ${\rm{KS}}_r$ conjecture.

\begin{theorem}\label{th0}{\rm ( see \cite[Theorem 1.4]{inter2} ) }
Let $\varepsilon>0$ and let $\mathbf{W}_1,\ldots,\mathbf{W}_m$ be independent random positive semidefinite Hermitian matrices in $\mathbb{C}^{d\times d}$ of rank 1 with finite support. If
\[
\sum_{i=1}^m \mathbb{E}\mathbf{W}_i = \mathbf{I}_d
\]
and
\[
\tr(\mathbb{E}\mathbf{W}_i)\leq \varepsilon,\quad \text{for all $i\in[m]$,}
\]
then
\begin{equation}\label{poss1}
\mathbb{P}\bigg[\bigg\|\sum_{i=1}^m\mathbf{W}_i\bigg\|\leq(1+\sqrt{\varepsilon})^2\bigg]>0.
\end{equation}
\end{theorem}

Theorem \ref{th1} directly follows from Theorem \ref{th0} and implies a positive solution to Weaver's ${\rm{KS}}_r$ conjecture.

\begin{theorem}\label{th1}{\rm ( see \cite[Corollary 1.5]{inter2} ) }
Let $r\geq 2$ be an integer. Assume that $\mathbf{X}_1,\ldots,\mathbf{X}_m\in\mathbb{C}^{d\times d}$ are positive semidefinite Hermitian matrices of rank at most $1$ such that
\begin{equation*}
\mathbf{X}_1+\cdots+\mathbf{X}_m=\mathbf{I}_d.
\end{equation*}
Let $\varepsilon:=\max\limits_{1\leq i\leq m} \tr(\mathbf{X}_i)$. Then there exists a partition $S_1\cup\cdots\cup S_r=[m]$ such that
\begin{equation}\label{thbound}
\bigg\|\sum\limits_{i\in S_j}^{}\mathbf{X}_i\bigg\|\leq \frac{1}{r}\cdot (1+\sqrt{r\varepsilon})^2 \quad\text{for all $j\in[r]$.}
\end{equation}
\end{theorem}

In particular, when $r=2$, if we set $\mathbf{X}_i=\frac{1}{\eta}\cdot \mathbf{u}_i\mathbf{u}_i^*$ and $\varepsilon=\frac{1}{\eta}$, then Theorem \ref{th1} implies Conjecture \ref{KS2} with any constant $\eta>(2+\sqrt{2})^2\approx 11.6569$.

\subsection{Related work}
Here we briefly introduce the previous improvements on Theorem \ref{th1}.

\subsubsection{The rank-one case}

When $r=2$, Bownik, Casazza, Marcus and Speegle \cite{bcms19} improves the upper bound in (\ref{thbound}) to  $\frac{1}{2}\cdot (\sqrt{1-2\varepsilon}+\sqrt{2\varepsilon})^2$ when $\varepsilon\leq\frac{1}{4}$.
This bound implies the same result in Conjecture \ref{KS2}, but with any constant  $\eta>4$. To our knowledge, this is the best known estimate on the constant $\eta$ in Conjecture \ref{KS2}.

Ravichandran and Leake in \cite{ravi2} adapted the method of interlacing families to directly prove Anderson's paving conjecture, which is well-known to be equivalent to Weaver's ${\rm{KS}}_r$ conjecture. They showed that, for any integer $r\geq2$ and any real number $0<\varepsilon\leq\frac{(r-1)^2}{r^2}$, if $\mathbf{A}\in\mathbb{C}^{m\times m}$ is a positive semidefinite matrix satisfying $\mathbf{0}\preceq\mathbf{A}\preceq\mathbf{I}_m$ and $\mathbf{A}(i,i)\leq\varepsilon$ for all $i$, then there exists a partition $S_1\cup\cdots\cup S_r=[m]$ such that
\begin{equation}\label{Ravichandran and Leake}
\|\mathbf{A}(S_j)\|\leq \frac{1}{r}\cdot (\sqrt{1-\frac{r\varepsilon}{r-1}}+\sqrt{r\varepsilon})^2\quad\text{for all $j\in[r]$.}
\end{equation}
Here, $\mathbf{A}(S)$  denotes the submatrix of $\mathbf{A}$ with rows and columns indexed by  $S\subset[m]$. Their result implies that the upper bound in (\ref{thbound}) can be improved to
\begin{equation}\label{Ravichandran and Leake2}
\bigg\|\sum\limits_{i\in S_j}^{}\mathbf{X}_i\bigg\|\leq \frac{1}{r}\cdot (\sqrt{1-\frac{r\varepsilon}{r-1}}+\sqrt{r\varepsilon})^2\quad\text{for all $j\in[r]$}
\end{equation}
when $\varepsilon\leq\frac{(r-1)^2}{r^2}$. To see this, we write $\mathbf{X}_i=\mathbf{u}_i\mathbf{u}_i^*$ for each $i\in[m]$ and set $\mathbf{A}=\mathbf{U}^*\mathbf{U}$ where $\mathbf{U}=[\mathbf{u}_1,\ldots,\mathbf{u}_m]\in\mathbb{C}^{d\times m}$. Then (\ref{Ravichandran and Leake2}) immediately follows since  $$\|\mathbf{A}(S_j)\|= \bigg\|\sum\limits_{i\in S_j}^{}\mathbf{X}_i\bigg\|$$  for each $j\in[r]$. When $r=2$, Ravichandran and Leake's bound in (\ref{Ravichandran and Leake2}) coincides with the estimate of Bownik et. al. from \cite{bcms19}. In  \cite{Ali}, Alishahi and Barzegar extended Ravichandran and Leake's result to the case of real stable polynomials and studied the paving property for strongly Rayleigh process.


\subsubsection{The higher-rank case}
In \cite{coh2016},  Cohen showed that  Theorem \ref{th1} holds for matrices  with higher ranks and the upper bound in (\ref{thbound}) still holds in this case. Br{\"a}nd{\'e}n also got rid of the rank constraints in \cite{Branden2} and extended Theorem \ref{th1} to the realm of hyperbolic polynomials. For each $\varepsilon>0$ and each integer $k>2$, let $\delta_{\varepsilon,k}$ be defined as
\begin{equation}\label{eq:delta}
\delta_{\varepsilon,k}:=\begin{cases} (\sqrt{1-\frac{\varepsilon}{k}}+\sqrt{\varepsilon})^2, & \text{if $\varepsilon\leq\frac{k}{k+1}$,}\\ 2+2\cdot\varepsilon(1-\frac{1}{k}), & \text{otherwise.} \end{cases}
\end{equation}
In the setting of Theorem \ref{th1}, Br{\"a}nd{\'e}n proved that, if the rank of each $\mathbf{X}_i$  is at most $k$, then the upper bound (\ref{thbound}) can be improved to $\frac{1}{r}\cdot\delta_{r\varepsilon,kr}$ when $kr>2$. For the case where $k=1$ and $r=2$, Br{\"a}nd{\'e}n also obtained the same bound with that of \cite{bcms19}. One can check that Ravichandran and Leake's bound
(\ref{Ravichandran and Leake2}) is better than ${\rm Br \ddot{a}nd \acute{e}n}$'s bound when $k=1$ and $r>2$, but ${\rm Br \ddot{a}nd \acute{e}n}$'s bound is more general since it is available for $k>1$.




\subsection{Our contribution}
In this paper we focus on extending the results in Theorem \ref{th0} and Theorem \ref{th1} to the higher-rank case. We first present the following theorem,  which improves the previous bounds from (\ref{poss1}) and from\cite[Theorem 6.1]{Branden2}.

\begin{theorem}\label{mth0}
Let $k\geq 2$ be an integer and let $\varepsilon\in(0, \frac{(k-1)^2}{k}]$. Let $\mathbf{W}_1,\ldots,\mathbf{W}_m$ be independent random positive semidefinite Hermitian matrices in $\mathbb{C}^{d\times d}$ with finite support. Suppose that	$\sum_{i=1}^m \mathbb{E}\mathbf{W}_i = \mathbf{I}_d$
and
\[
\operatorname{tr}(\mathbb{E}\mathbf{W}_i)\leq \varepsilon,\,\,\operatorname{rank}(\mathbb{E}\mathbf{W}_i)\leq k\quad \text{for all $i\in[m]$} .
\]
Then,
\[
\mathbb{P}\bigg[\bigg\|\sum_{i=1}^m\mathbf{W}_i\bigg\|\leq\bigg(\sqrt{1-\frac{\varepsilon}{k-1}}+\sqrt{\varepsilon}\bigg)^2\bigg]>0.
\]
\end{theorem}

We immediately obtain the following corollary with a similar argument in \cite[Corollary 1.5]{inter2}.

\begin{corollary}\label{mth1}
Let $r\geq 2$ and $k\geq 1$ be integers. Assume that $\mathbf{X}_1,\ldots,\mathbf{X}_m\in\mathbb{C}^{d\times d}$ are positive semidefinite Hermitian matrices of rank at most $k$ such that
\begin{equation}\label{eq:sumid}
\mathbf{X}_1+\cdots+\mathbf{X}_m=\mathbf{I}_d.
\end{equation}
Let $\varepsilon:=\max\limits_{1\leq i\leq m} \tr(\mathbf{X}_i)$. If $\varepsilon\leq\frac{(kr-1)^2}{kr^2}$, then there exists a partition $S_1\cup\cdots\cup S_r=[m]$ such that
\begin{equation}\label{mth1bound}
\bigg\|\sum\limits_{i\in S_j}^{}\mathbf{X}_i\bigg\|\leq \frac{1}{r}\cdot\bigg(\sqrt{1-\frac{r\varepsilon}{kr-1}}+\sqrt{r\varepsilon}\bigg)^2\quad \text{for all $j\in[r]$.}
\end{equation}
\end{corollary}


\begin{proof}
For each $i\in[m]$, let $\mathbf{W}_i$ be a random matrix that takes the following matrices of size $rd\times rd$ with equal probability:
\begin{equation*}
\mathbf{W}_{i,1}:=\begin{bmatrix}
r\cdot\mathbf{X}_i & & &\\
&0&&\\
&&\ddots&\\
&&&0
\end{bmatrix},
\mathbf{W}_{i,2}:=\begin{bmatrix}
0 & & &\\
&r\cdot\mathbf{X}_i&&\\
&&\ddots&\\
&&&0
\end{bmatrix},
\ldots,
\mathbf{W}_{i,r}:=\begin{bmatrix}
0 & & &\\
&0&&\\
&&\ddots&\\
&&&r\cdot\mathbf{X}_i
\end{bmatrix}.
\end{equation*}
A simple calculation  shows that
\begin{equation*}
\mathbb{E}\mathbf{W}_i=\begin{bmatrix}
\mathbf{X}_i & & &\\
&\mathbf{X}_i&&\\
&&\ddots&\\
&&&\mathbf{X}_i
\end{bmatrix}\in\mathbb{C}^{rd\times rd},\quad \text{for all $i\in[m]$.}
\end{equation*}
This gives
\begin{equation*}
\operatorname{tr}(\mathbb{E}\mathbf{W}_i)\leq r\varepsilon\quad \text{and}\quad \rank(\mathbb{E}\mathbf{W}_i)\leq kr\quad \text{for all $i\in[m]$}.
\end{equation*}
We also have $\sum_{i=1}^m \mathbb{E}\mathbf{W}_i = \mathbf{I}_{rd}$. By Theorem \ref{mth0}, we obtain
\begin{equation}\label{eq:corollary1}
\mathbb{P}\bigg[\bigg\|\sum_{i=1}^{m}\mathbf{W}_i\bigg\|\leq\bigg(\sqrt{1-\frac{r\varepsilon}{kr-1}}+\sqrt{r\varepsilon}\bigg)^2\bigg]>0.
\end{equation}
Hence, for each $i\in [m]$ there exists $j_i\in [r]$  so that
\[
\bigg\|\sum_{i=1}^{m}\mathbf{W}_{i,j_i}\bigg\|\leq\bigg(\sqrt{1-\frac{r\varepsilon}{kr-1}}+\sqrt{r\varepsilon}\bigg)^2.
\]
For each $j\in[r]$, set $S_j:=\{i\in[m]:j_i=j\}$. Then $S_1,\ldots,S_r$ form a partition of $[m]$, and the (\ref{eq:corollary1}) gives
\[
\bigg\|\sum_{i\in S_j} \mathbf{X}_i\bigg\|\leq \frac{1}{r}\bigg(\sqrt{1-\frac{r\varepsilon}{kr-1}}+\sqrt{r\varepsilon}\bigg)^2\quad \text{for all $j\in[r]$}.
\]
\end{proof}
\begin{remark}
Motivated by the argument  in   \cite[Proposition 2.2]{FY19},  we can relax the condition (\ref{eq:sumid}) in Corollary \ref{mth1}  to $\sum_{i=1}^m\mathbf{X}_i\preceq \mathbf{I}_d$. More specifically,   we can find rank-one matrices $\{\vv_j\vv_j^*\}_{1\leq j\leq M}$ with trace at most $\varepsilon$ such that $\sum_{i=1}^m\mathbf{X}_i + \sum_{j=1}^M \vv_j\vv_j^* = \mathbf{I}_d$. Then there exists an partition of $[m+M]$ satisfying (\ref{mth1bound}) by Corollary \ref{mth1}. We can get a desired partition of $[m]$ by restricting each subset in the partition of $[m+M]$ to $[m]$.
\end{remark}
Our bound in (\ref{mth1bound}) coincides with that of \cite{ravi2} for each $r\geq2$ when $k=1$. In particular, our bound is also the same with that of \cite{bcms19} when $r=2$ and $k=1$. For the case when $k\geq2$, our bound (\ref{mth1bound}) slightly improves Br{\"a}nd{\'e}n's bound, i.e., $\frac{1}{r}\cdot\delta_{r\varepsilon,kr}=\frac{1}{r}\cdot(\sqrt{1-\frac{\varepsilon}{k}}+\sqrt{r\varepsilon})^2$.
We summarize the related works in Table \ref{tab:ks}.


 \makeatletter\def\@captype{table}\makeatother
\begin{center}
\caption{The estimate on the paving bound in Corollary \ref{mth1}}\label{tab:ks}
\begin{tabular}{@{}ll@{}}
\toprule
 The value of $k$ and $r$                  &  Paving bound  in  (\ref{mth1bound}) \\ \midrule
$k=1,r=2$ & $\frac{1}{2}\cdot (\sqrt{1-2\varepsilon}+\sqrt{2\varepsilon})^2$ for $\varepsilon\leq\frac{1}{4}$ \cite{bcms19,ravi2,Branden2} \\ \midrule
\multirow{2}{*}{$k=1,r\geq 2$}  &  $\frac{1}{r}\cdot(1+\sqrt{r\epsilon})^2$ \cite{inter2}\\
                  &  $\frac{1}{r}\cdot (\sqrt{1-\frac{r\varepsilon}{r-1}}+\sqrt{r\varepsilon})^2$ for $\varepsilon\leq\frac{(r-1)^2}{r^2}$ \cite{ravi2}\\ \midrule
\multirow{4}{*}{$k\geq 1,r\geq 2$} &  $\frac{1}{r}\cdot(1+\sqrt{r\epsilon})^2$ \cite{coh2016}\\
                  & $\frac{1}{r}\cdot(\sqrt{1-\frac{\varepsilon}{k}}+\sqrt{r\varepsilon})^2$ for $\varepsilon \leq \frac{k}{kr+1}$\cite{Branden2} \\
                  & $\frac{1}{r}\cdot(2+2\cdot r\varepsilon(1-\frac{1}{kr}))$ for $\varepsilon > \frac{k}{kr+1}$\cite{Branden2} \\
                  & $ \frac{1}{r}\cdot(\sqrt{1-\frac{r\varepsilon}{kr-1}}+\sqrt{r\varepsilon})^2 $ for $\varepsilon\leq\frac{(kr-1)^2}{kr^2}$ (Corollary \ref{mth1}) \\ \bottomrule
\end{tabular}
\end{center}

 	We next provide an application of Corollary \ref{mth1}, which  specifies a simultaneous paving bound for  multiple positive semidefinite Hermitian matrices. In \cite{ravi3}, Ravichandran and Srivastava proved a simultaneous paving  bound for a tuple of zero-diagonal Hermitian matrices. Therefore, the following corollary serves as a counterpart of \cite[Theorem 1.1]{ravi3}. The result also coincides with \cite[Theorem 2]{ravi2} when paving just one matrix.
 	
 	\begin{corollary}
 		Let $r\geq 2$ and $k\geq 1$ be integers.
 		Assume that $\mathbf{A}_1,\ldots,\mathbf{A}_k\in\mathbb{C}^{m\times m}$ are positive semidefinite Hermitian matrices satisfying $\mathbf{0}\preceq\mathbf{A}_i\preceq\mathbf{I}_m$ for $1\leq i\leq k$.
 		Let $\alpha:= \max\limits_{1\leq i\leq k}\max\limits_{1\leq l\leq m}\mathbf{A}_i(l,l)$. If $\alpha\leq \frac{(kr-1)^2}{k^2r^2}$, then there exists a partition $S_1\cup\cdots\cup S_r=[m]$ such that
 		\[
 			\|\mathbf{A}_i(S_j)\|\leq \bigg(\sqrt{\frac{1}{r}-\frac{k\alpha}{kr-1}}+\sqrt{k\alpha}\bigg)^2\quad \text{for all $i\in[k]$ and $j\in[r]$.}.
 		\]
 	\end{corollary}
 	\begin{proof}
 		For $1\leq i\leq k$, let the vectors $\{\mathbf{u}_{i,l}\}_{l\in[m]}\subset\mathbb{C}^d$ satisfy $\mathbf{A}_i = (\left<\mathbf{u}_{i,l_1},\mathbf{u}_{i,l_2}\right>)_{1\leq l_1,l_2\leq m}$, and then  we have
 		$
 			\sum_{l=1}^m \vu_{i,l}\vu_{i,l}^*\preceq \mathbf{I}_d
 		$
 		 		and $\|\mathbf{u}_{i,l}\|^2\leq \alpha$
 		for $1\leq l \leq m$.
 		For $1\leq l\leq m$, define a block diagonal matrix
 		\begin{equation*}
 			\mathbf{X}_{l}:=\begin{bmatrix}
 			\vu_{1,l}\vu_{1,l}^* & & &\\
 			&	\vu_{2,l}\vu_{2,l}^*&&\\
 			&&\ddots&\\
 			&&&	\vu_{k,l}\vu_{k,l}^*
 		\end{bmatrix}\in\mathbb{C}^{kd\times kd}.
 		\end{equation*}
 		Then $  \max\limits_{1\leq l\leq m}\tr(\mathbf{X}_l)\leq k  \alpha \leq  \frac{(kr-1)^2}{kr^2}$ .
 		Note that
 		$
 			\sum_{l=1}^m\mathbf{X}_l\preceq \mathbf{I}_{kd}.
 		$ By Collolary \ref{mth1}, there exists a partition $S_1\cup\cdots\cup S_r=[m]$ such that
 		\begin{equation}\label{eq:sp1}
 			\bigg\|\sum_{l\in S_j}\mathbf{X}_l\bigg\| \leq  \frac{1}{r}\cdot\bigg(\sqrt{1-\frac{rk\alpha}{kr-1}}+\sqrt{rk\alpha}\bigg)^2 \quad \text{for all  $j\in[r]$.}.
 		\end{equation}
 		Note that
 		\begin{equation*}
 			\sum_{l\in S_j}\mathbf{X}_l=\begin{bmatrix}
 				\sum\limits_{l\in S_j}\vu_{1,l}\vu_{1,l}^* & & &\\
 				&		\sum\limits_{l\in S_j}\vu_{2,l}\vu_{2,l}^*&&\\
 				&&\ddots&\\
 				&&&		\sum\limits_{l\in S_j}\vu_{k,l}\vu_{k,l}^*
 			\end{bmatrix}.
 		\end{equation*}
 		Then  we have
 		\begin{equation}\label{eq:sp2}
 			\bigg\|\sum_{l\in S_j}\mathbf{X}_l\bigg\| = \max_{1\leq i\leq k}\bigg\|	\sum\limits_{l\in S_j}\vu_{i,l}\vu_{i,l}^*\bigg\|  = \max_{1\leq i\leq k}\|\mathbf{A}_i(S_j)\|.
 		\end{equation}
 		Combining (\ref{eq:sp1}) and (\ref{eq:sp2}), we arrive at the conclusion.
 	\end{proof}
\subsection{Our techniques}

To introduce our techniques, let us briefly recall the proof of Theorem \ref{th0} and how \cite{coh2016} and \cite{Branden2} extended Theorem \ref{th0} to the higher-rank case.

For the rank-one case, let $\{\mathbf{W}_i\}_{1\leq i\leq m}$ be as defined in Theorem \ref{th0}. For $1\leq i\leq m$, let the support of $\mathbf{W}_i$ be
$
W_i:=\{\mathbf{W}_{i,1},\ldots,\mathbf{W}_{i,l_i}\}
$. In \cite{inter2}, Marcus, Spielman and Srivastava  showed that the characteristic polynomials of
$\{\sum_{i=1}^m\mathbf{W}_{i,j_i}:1\leq j_i\leq l_i ,\,i=1,\ldots,m\}$
form a so-called interlacing family. This implies that there exists a polynomial in this family whose largest root is at most that of the expectation of the characteristic polynomial of $\sum_{i=1}^m\mathbf{W}_i$. Hence, it is enough to estimate the largest root of this expected characteristic polynomial. This expected characteristic polynomial is referred to as the mixed characteristic polynomial:

\begin{definition}\label{def-mix}{\rm (see \cite{inter2})}
Given $\mathbf{X}_1,\ldots,\mathbf{X}_m\in\mathbb{C}^{d\times d}$, the \emph{mixed
characteristic polynomial} of $\mathbf{X}_1,\ldots,\mathbf{X}_m$ is defined as
\begin{equation}\label{def-mix3}
\mu[\mathbf{X}_1,\ldots,\mathbf{X}_m](x) := \prod_{i=1}^m(1-\partial_{z_i})\det[x\cdot\mathbf{I}_d+\sum_{i=1}^mz_i\mathbf{X}_i]|_{z_1=\cdots=z_m=0}.
\end{equation}
\end{definition}

Assume that $\mathbf{W}_1,\ldots,\mathbf{W}_m$ are independent random matrices of
rank one in $\mathbb{C}^{d\times d}$ satisfying $\mathbb{E}
\mathbf{W}_i=\mathbf{X}_i$ for each $i\in[m]$.
 Marcus, Spielman and Srivastava in \cite[Theorem 4.1]{inter2} showed that
\begin{equation}\label{def-mix2}
\mu[\mathbf{X}_1,\ldots,\mathbf{X}_m](x)=\mathbb{E}\ \det[x\cdot\mathbf{I}_d-\sum\limits_{i=1}^{m}\mathbf{W}_i].
\end{equation}
Based on the above formula, they employed an argument of the barrier  function that was
developed in \cite{inter0} to estimate the largest root of the mixed characteristic
polynomials.

For the higher-rank case, instead of studying the characteristic polynomials  of
$\{\sum_{i=1}^m\mathbf{W}_{i,j_i}:1\leq j_i\leq l_i ,\,i=1,\ldots,m\}$ , the authors of \cite{coh2016} and \cite{Branden2}
concentrated their attention on the mixed characteristic polynomials of
$\mathbf{W}_{1,j_1},\ldots,\mathbf{W}_{m,j_m}$. They showed that this family of polynomials also
form an interlacing family. Furthermore, they proved that, for any positive
semidefinite Hermitian matrices
$\mathbf{X}_1,\ldots,\mathbf{X}_m\in\mathbb{C}^{d\times d}$, the operator norm of
$\sum_{i=1}^m\mathbf{X}_i$  is upper bounded by the largest root of
$\mu[\mathbf{X}_1,\ldots,\mathbf{X}_m]$, i.e.
\[
	\bigg\|\sum_{i=1}^m\mathbf{X}_i\bigg\|\leq \operatorname{maxroot}\ \mu[\mathbf{X}_1,\ldots,\mathbf{X}_m].
\] Hence, the original problem is reduced to estimating the largest root of the expectation of the mixed characteristic polynomials, which can be done with a similar argument of barrier function.

To prove Theorem \ref{mth0}, we follow the above framework of \cite{coh2016} and \cite{Branden2}. Our main technique is that we derive a new formula for the mixed characteristic polynomials (see Theorem \ref{th-mix}).
 Based on this new formula and the method of barrier function,  we obtain an improved estimate on the largest root of the mixed characteristic polynomials. We state it as follows.

\begin{theorem}\label{maxroot-mu}
Assume that $\mathbf{X}_1,\ldots,\mathbf{X}_m\in\mathbb{C}^{d\times d}$ are positive semidefinite Hermitian matrices of rank at most $k$ such that $\sum\limits_{i=1}^{m}\mathbf{X}_i\preceq\mathbf{I}_d$. Let $\varepsilon:=\max\limits_{1\leq i\leq m}^{}\tr(\mathbf{X}_i)$. If $\varepsilon\leq\frac{(k-1)^2}{k}$,
then we have
\begin{equation}\label{maxroot-mu1}
\operatorname{maxroot}\ \mu[\mathbf{X}_1,\ldots,\mathbf{X}_m]\leq \bigg(\sqrt{1-\frac{\varepsilon}{k-1}}+\sqrt{\varepsilon}\bigg)^2.
\end{equation}
\end{theorem}

%

\subsection{Organization}

This paper is organized as follows. After   introducing some useful notation and lemmas in Section \ref{s:pre},  we introduce the definition  of  $(k,m)$-characteristic polynomials, showing the connection between   $(k,m)$-characteristic polynomials and the mixed characteristic polynomials in Section \ref{s:new}. In Section \ref{s:barrier}, we use the method of barrier function to present a proof of Theorem \ref{maxroot-mu}.  The  proof of Theorem \ref{mth0} is presented in Section \ref{s:proof}.

\section{Preliminaries}\label{s:pre}

\subsection{Notations}

We first introduce some notations. For a vector $\mathbf{x}\in\mathbb{C}^m$, we let $\|\mathbf{x}\|$ denote its Euclidean 2-norm. For a matrix $\mathbf{B}\in\mathbb{C}^{m\times m}$, we use $\|\mathbf{B}\|=\max_{\|\mathbf{x}\|=1}^{}\|\mathbf{Bx}\|$ to denote its operator norm. 
We write $\partial_{z_i}$ to indicate the partial differential $\partial\slash\partial_{z_i}$.
For each $t\in[m]$, let $\mathbf{e}_t\in\mathbb{R}^m$ denote the vector whose $t$-th entry equals $1$ and the rest entries equal to $0$. For a polynomial $p\in\mathbb{R}[z]$ with real roots, we use $\operatorname{maxroot}\ p$ and $\operatorname{minroot}\ p$ to denote the maximum and minimum root of $p$ respectively.

For an integer $m$, we use $[m]$ to denote the set $\{1,2,\ldots,m\}$.
For any two positive integers $k$ and $m$, we call $\mathcal{S}=( S_1,\ldots, S_k)\in[m]^k$ an $k$-partition of $[m]$ if $S_1,\ldots, S_k$ are disjoint and $S_1\cup\cdots\cup S_k=[m]$.  We use the notation $\mathcal{P}_k(m)$ to denote the set of all $k$-partitions of $[m]$.

Given a matrix $\mathbf{A}\in\mathbb{C}^{m\times m}$, for a subset $S\subset[m]$,  we
use $\mathbf{A}(S)$ to denote the principal submatrix of $\mathbf{A}$ whose rows and
columns are indexed by $S$. Let $k\geq 1$ be an integer. Given a matrix
$\mathbf{A}\in\mathbb{C}^{km\times km}$, for $\mathcal{S}=(S_1,\ldots,S_k)\in[m]^k$,
we use $\mathbf{A}(\mathcal{S})$ to denote the principal submatrix
$\mathbf{A}(S_1\cup(m+S_2)\cup\ldots\cup((k-1)\cdot m+S_k))$. For example, let
$m=4,k=3$ and let $S_1=\{1,2\},S_2=\{2,3,4\},S_3 = \{3\}$. If we set
$\mathcal{S}=(S_1,S_2,S_3)$, then for a matrix $\mathbf{A}\in\mathbb{C}^{12\times
12}$ the principal submatrix $\mathbf{A}(\mathcal{S})\in\mathbb{C}^{6\times 6}$ is
composed of the shaded part in Figure ~\ref{fig:ps}.
\begin{figure}[ht]	
	\centering
	\includegraphics[scale=0.4]{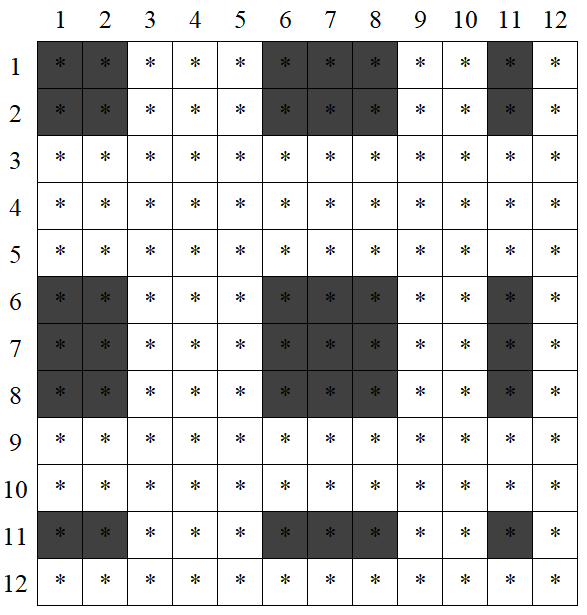}
	\caption{The principal submatrix $\mathbf{A}(\mathcal{S})\in\mathbb{C}^{6\times 6}$ for $m=4,k=3$,
		   $\mathcal{S}=(\{1,2\},\{2,3,4\}, \{3\})$, and $\mathbf{A}\in\mathbb{C}^{12\times 12.}$}
	\label{fig:ps}
\end{figure}

\subsection{Interlacing families} The method  of interlacing families is a powerful tool to show the existence of some combinatorial objects. Marcus, Spielman and Srivastava employed this tool to prove the existence of bipartite Ramanujan graphs of all sizes and degrees, solved the Kadison-Singer problem and proved a sharper restricted invertibility \cite{inter1,inter2,inter3,inter4}.

Here we recall the definition and related results of interlacing families from \cite{inter2}. Throughout this paper, we say that a univariate polynomial is \emph{real-rooted} if all of its coefficients and roots are real.
\begin{definition}{\rm (see \cite[Definition 3.1]{inter2} ) }
We say a real-rooted polynomial $g(x) =\alpha_0\prod_{i=1}^{n-1}(x-\alpha_i)$ \emph{interlaces} a real-rooted polynomial $f(x) =\beta_0\prod_{i=1}^{n}(x - \beta_i)$ if $\beta_1\leq\alpha_1\leq\beta_2\leq\alpha_2\leq\cdots\leq \alpha_{n-1}\leq \beta_n $.
For polynomials $f_1,\ldots,f_k$, if there exists a polynomial $g$ that interlaces $f_i$ for each $i$, then we say that $f_1,\ldots,f_k$ have a \emph{common interlacing}.
\end{definition}

We next introduce the definition of  interlacing families.

\begin{definition}{\rm (see \cite[Definition 3.3]{inter2} ) }
Assume that $S_1,\ldots,S_m$ are finite sets.  For every assignment $s_1,\ldots,s_m\in S_1\times\cdots\times S_m$, let $f_{s_1,\ldots,s_m}(x)$ be a real-rooted degree $n$ polynomial with positive leading coefficient. For each $k<m$ and each partial assignment $s_1,\ldots,s_k\in S_1\times\cdots\times S_k$, we define
\begin{equation*}
f_{s_1,\ldots,s_k}:=\sum\limits_{s_{k+1}\in S_{k+1},\ldots,s_m\in S_{m}}^{}f_{s_1,\ldots,s_k,s_{k+1},\ldots,s_m}.
\end{equation*}
We also define
\begin{equation*}
f_{\emptyset}:=\sum\limits_{s_{1}\in S_{1},\ldots,s_m\in S_{m}}^{}f_{s_1,\ldots,s_m}.
\end{equation*}
We say that the polynomials $\{f_{s_1,\ldots,s_m}: s_1,\ldots,s_m\in S_1\times\cdots\times S_m\}$ form an \emph{interlacing family} if for all integers $k=0,\ldots,m - 1$ and all partial assignments $s_1,\ldots,s_k\in S_1\times\cdots\times S_k$, the polynomials $\{f_{s_1,\ldots,s_k,t}\}_{t\in S_{k+1}}$ have a common interlacing.
\end{definition}

We state the main property of interlacing families as the following lemma.
\begin{lemma}\label{interlacing2}{\rm (see \cite[Theorem 3.4]{inter2} ) }
Assume that  $S_1,\ldots,S_m$ be finite sets and let $\{f_{s_1,\ldots,s_m}: s_1,\ldots,s_m\in S_1\times\cdots\times S_m\}$ be an interlacing family. Then there exists some $s'_1,\ldots,s'_m\in S_1\times\cdots\times S_m$ such that the largest root of $f_{s'_1,\ldots,s'_m} $ is upper bounded by the largest root of $f_{\emptyset}$.
\end{lemma}

\subsection{Real stable polynomials} Through our analysis, we exploit the notion of real stable polynomials, which can be viewed as a  multivariate generalization of real-rooted polynomials. For more details, see \cite{Branden01,wag}.

We first introduce the definition of real stable polynomials.

\begin{definition}
A  polynomial $p\in\mathbb{R}[z_1,\ldots, z_m]$ is \emph{real stable} if  $p(z_1,\ldots,z_m)\neq 0$ for all $(z_1,\ldots,z_n)\in\mathbb{C}^n$ with $\mathbf{Im}(z_i)>0$ and $i\in[m]$.
\end{definition}

To show the polynomial we concern in this paper is real stable, we need the following lemma.
\begin{lemma}\label{real stable}{\rm (see \cite[Proposition 2.4]{Branden} ) }
For any positive semidefinite Hermitian matrices $\mathbf{A}_1,\ldots,\mathbf{A}_m\in\mathbb{C}^{d\times d}$ and any Hermitian matrix $\mathbf{B}\in\mathbb{C}^{d\times d}$, the polynomial
\begin{equation}
\det[\mathbf{A}_1 z_1+\cdots+\mathbf{A}_m z_m+\mathbf{B}]\in\mathbb{R}[z_1,\ldots,z_m]
\end{equation}
is either real stable or identically zero.
\end{lemma}

Real stability can be preserved under some certain transformations. In our proof, we use some real stability preservers to reduce  multivariate polynomials to a univariate one which is a real-rooted polynomial. 
\begin{lemma}\label{real stable2}{\rm (see \cite[Lemma 2.4]{wag} ) }
If $p \in \mathbb{R}[z_1,\ldots, z_m]$ is real stable and $a\in\mathbb{R}$, then the following polynomials are also real stable:
\begin{itemize}
\item $p(a, z_2 ,\ldots, z_m)\in\mathbb{R}[z_2 , . . . , z_m]$,
\item $p(z_1,\ldots,z_m)|_{z_i = z_j}$, for all $\{i,j\}\subset [m]$,
\item $\partial_{z_i}p(z_1,\ldots,z_m)$, for all  $i\in[m]$.
\end{itemize}
\end{lemma}

\subsection{The mixed characteristic polynomial}

The mixed characteristic polynomial plays a key role in the proof of Theorem \ref{th0}. In this subsection we introduce some basic properties of mixed characteristic polynomials which is useful in our proof of Theorem \ref{mth0}.

\begin{lemma}\label{mix-th1}{\rm {(see \cite[Theorem 5.2]{Branden2} and \cite{coh2016})} }
Assume that $\mathbf{X}_1,\ldots,\mathbf{X}_m\in\mathbb{C}^{d\times d}$ are positive semidefinite Hermitian matrices.  Then we have
\begin{equation}\label{maxroot}
\bigg\|\sum\limits_{i=1}^{m}\mathbf{X}_i\bigg\|\leq    \operatorname{maxroot}\ \mu[\mathbf{X}_1,\ldots,\mathbf{X}_m],
\end{equation}
where  $\mu[\mathbf{X}_1,\ldots,\mathbf{X}_m]$ is defined in (\ref{def-mix3}).

\end{lemma}

\begin{lemma}\label{mix-th2}{\rm {(see \cite[Theorem 3.5]{Branden2} and \cite{coh2016})} }
Assume that $\mathbf{W}_1,\ldots,\mathbf{W}_m$ are independent random positive semidefinite Hermitian matrices in $\mathbb{C}^{d\times d}$ with finite support. For each $i\in[m]$ let $W_i:=\{\mathbf{W}_{i,1},\ldots,\mathbf{W}_{i,l_i}\}$ be the support of $\mathbf{W}_i$. Then, the mixed characteristic polynomials
\begin{equation}
\mu[\mathbf{W}_{1,j_1},\ldots,\mathbf{W}_{m,j_m}](x),\quad 1\leq j_i\leq l_i,\,\, i=1,\ldots,m
\end{equation}
form an interlacing family.
\end{lemma}

It is also pointed out by Br{\"a}nd{\'e}n in \cite{Branden2} that the mixed characteristic polynomial $\mu[\mathbf{X}_1,\ldots,\mathbf{X}_m](x)$ is affine linear in each $\mathbf{X}_i$, i.e., for all $\alpha\in\mathbb{R}$ and all $i\in[m]$:
\begin{equation}
\begin{aligned}
&\quad\ \ \mu[\mathbf{X}_1,\ldots,(1-\alpha)\cdot\mathbf{X}_i+\alpha\cdot \mathbf{X}_i' , \ldots,\mathbf{X}_m ](x)\\
&=(1-\alpha)\cdot \mu[\mathbf{X}_1,\ldots,\mathbf{X}_i, \ldots,\mathbf{X}_m ](x)+\alpha\cdot \mu[\mathbf{X}_1,\ldots,\mathbf{X}_i ', \ldots,\mathbf{X}_m ](x),
\end{aligned}
\end{equation}
where  $\mathbf{X}_1,\ldots,\mathbf{X}_m, \mathbf{X}_i'\in\mathbb{C}^{d\times d}$. 
Hence, if $\mathbf{W}_1,\ldots,\mathbf{W}_m$ are independent random 
 matrices with finite support, then we have
\begin{equation}\label{mix-th3}
\mathbb{E}\ \mu[\mathbf{W}_1,\ldots,\mathbf{W}_m](x)=\mu[\mathbb{E}\mathbf{W}_1,\ldots,\mathbb{E}\mathbf{W}_m](x).
\end{equation}

\section{A new formula for mixed characteristic polynomials}\label{s:new}

For positive integers  $k$ and $m$, we first  introduce the $(k,m)$-determinant  and
$(k,m)$-characteristic polynomial of a matrix $\mathbf{A}\in\mathbb{C}^{km\times
km}$. Then we show some connections between   $(k,m)$-characteristic polynomials and
mixed characteristic polynomials. Recall that we use $\mathcal{P}_k(m)$ to denote the
set of all $k$-partitions of $[m]$. Also recall that for each
$\mathbf{A}\in\mathbb{C}^{km\times km}$ and each $\mathcal{S}=( S_1,\ldots, S_k)\in
[m]^k$ we use $\mathbf{A}( \mathcal{S}) $ to denote the principal submatrix of
$\mathbf{A}$ with rows and columns indexed by $S_1\cup (m+
S_2)\cup\cdots\cup((k-1)\cdot m+ S_k)$.



\begin{definition}\label{def1}
Let $k,m$ be two positive integers. For any matrix $\mathbf{A}\in\mathbb{C}^{km\times km}$, the \emph{$(k,m)$-determinant}  $D_{k,m}[\mathbf{A}]$ of $\mathbf{A}$ is defined as
\begin{equation*}
D_{k,m}[\mathbf{A}]:=\sum\limits_{ \mathcal{S}\in \mathcal{P}_k(m)}^{}\det[ \mathbf{A}( \mathcal{S})].
\end{equation*}
The \emph{$(k,m)$-characteristic polynomial}  $\psi_{k,m}[\mathbf{A}](x)$ of $\mathbf{A}$ is defined as
\begin{equation}\label{eq:AS3}
\psi_{k,m}[\mathbf{A}](x):=\frac{1}{k^m}\cdot D_{k,m}[x\cdot \mathbf{I}_{km}-\mathbf{A}].
\end{equation}
\end{definition}

\begin{remark}
A simple calculation shows  that $\psi_{k,m}[\mathbf{A}](x)$ is a monic polynomial of
degree $m$. For $k=1$, we have $D_{1,m}[\mathbf{A}]=\det[\mathbf{A}]$, and
$\psi_{1,m}[\mathbf{A}](x)$ is the regular characteristic polynomial of
$\mathbf{A}\in\mathbb{C}^{m\times m}$. If we take $m=1$, then
$\psi_{k,1}[\mathbf{A}](x)$ is a multiple of the $(k-1)$-th order derivative of the
characteristic polynomial of $\mathbf{A}\in\mathbb{C}^{k\times k}$ (see Proposition
\ref{prop-1}).
\end{remark}

We next introduce some connections among $D_{k,m}[\mathbf{A}],
\psi_{k,m}[\mathbf{A}](x)$ and other generalizations of the determinant and the
characteristic polynomial in the literature.

\begin{enumerate}[1.]

\item (Borcea-Br{\"a}nd{\'e}n \cite{Branden})  For a $k$-tuple
    $(\mathbf{A}_1,\ldots,\mathbf{A}_k)$ of matrices in $\mathbb{C}^{m\times m}$,
    the \emph{mixed determinant}, first introduced by Borcea and Br{\"a}nd{\'e}n
    in \cite{Branden}, is defined as
	\[
	D[\mathbf{A}_1,\ldots,\mathbf{A}_k] :=\sum_{(S_1,\ldots,S_k)\in \mathcal{P}_k(m)}\prod_{i=1}^k\det[\mathbf{A}_i(S_i)].
	\] Note that $D(x\cdot\mathbf{I}_m,-\mathbf{B})=\det(x\cdot
\mathbf{I}_m-\mathbf{B})$  is the regular characteristic polynomial of
$\mathbf{B}\in\mathbb{C}^{m\times m}$. In \cite{Branden}, Borcea and
Br{\"a}nd{\'e}n  used the notion of the mixed determinant to prove a stronger
version of Johnson's Conjecture. We remark here that the $(k,m)$-determinant
introduced in Definition \ref{def1} can be viewed as a generalization of the
mixed determinant by the following identity:
\begin{equation}\label{determinant}
D_{k,m}[\mathbf{A}]\,\,=\,\,D[\mathbf{A}_1,\ldots,\mathbf{A}_k],
\end{equation}
where $\mathbf{A}=\diag(\mathbf{A}_1,\mathbf{A}_2,\ldots,\mathbf{A}_k)\in\mathbb{C}^{km\times km}$ is a block diagonal matrix.

\item (Ravichandran-Leake \cite{ravi2})  For a matrix $\mathbf{B}\in \mathbb{C}^{m\times m}$,  the \emph{$k$-characteristic polynomial} of the matrix $\mathbf{B}$ is defined as (see \cite[Proposition 2]{ravi2})
\begin{equation}\label{eq:kcp}
\chi_k[\mathbf{B}](x) := D[\underbrace{x\cdot \mathbf{I}_m-\mathbf{B},\ldots,x\cdot \mathbf{I}_m-\mathbf{B}}_k].
\end{equation}
Ravichandran and Leake used the $k$-characteristic polynomial to prove Anderson's paving formulation of Kadison-Singer problem \cite{ravi2}.
It is easy to see the relationship between  $\chi_k$ and $\psi_{k,m}$ :
\begin{equation}\label{eq:RL}
\chi_k[\mathbf{B}](x)=k^m\cdot \psi_{k,m}[\diag(\underbrace{\mathbf{B},\ldots,\mathbf{B}}_k)](x).
\end{equation}

\item (Ravichandran-Srivastava \cite{ravi3})  For a $k$-tuple
    $(\mathbf{A}_1,\ldots,\mathbf{A}_k)$ of matrices in $\mathbb{C}^{m\times m}$,
    the \emph{mixed determinantal polynomial} is defined as
\begin{equation}\label{eq:mdp}
	\chi[\mathbf{A}_1,\ldots,\mathbf{A}_k](x):=\frac{1}{k^m}\cdot D[x\cdot \mathbf{I}_m-\mathbf{A}_1,\ldots,x\cdot \mathbf{I}_m-\mathbf{A}_k].
\end{equation}
Ravichandran and Srivastava used the mixed determinantal polynomial to provide a
 simultaneous paving of a $k$-tuple of zero-diagonal  Hermitian matrices
 \cite{ravi2}. According to  (\ref{determinant}), we immediately have the
 following identity:
\begin{equation}\label{eq:RS}
	\chi[\mathbf{A}_1,\ldots,\mathbf{A}_k](x)=\psi_{k,m} [\diag(\mathbf{A}_1,\ldots,\mathbf{A}_k)](x).
\end{equation}
Hence, the $(k,m)$-characteristic polynomial can be considered as  a generalization of the mixed determinantal polynomial.
\end{enumerate}

We next provide several basic properties about the $(k,m)$-characteristic polynomial.
The first one provides an alternative expression of  $(k,m)$-characteristic polynomial.
\begin{prop}\label{prop-1}
Let $k,m$ be two positive integers. Let $$ \mathbf{Z}_k:=\diag(\underbrace{\mathbf{z},\ldots,\mathbf{z}}_{k})\quad  {\rm and} \quad \mathbf{z}:=(z_1,\ldots,z_m).$$ Then, for any matrix $\mathbf{A}\in\mathbb{C}^{km\times km}$, we have
\begin{equation}\label{eq:AS2}
\psi_{k,m}[\mathbf{A}](x)=\frac{1}{(k!)^m}\cdot\prod\limits_{i=1}^{m}\partial_{z_i}^{k-1} \cdot\det[\mathbf{Z}_k-\mathbf{A}] |_{z_1=\cdots=z_m=x}.
\end{equation}
\end{prop}
\begin{remark}
	Proposition \ref{prop-1}  can be considered as  a  generalization of  \cite[Proposition 1]{ravi2} and \cite[Proposition 2.3]{ravi3}, which express the  $k$-characteristic polynomial  and the mixed determinantal polynomial both in differential formulas. In \cite{ravi2},
 Ravichandran and  Leake showed that
	\[
		\chi_k[\mathbf{B}](x)= \prod_{i=1}^m\partial_{z_i}^{k-1}\det[\mathbf{Z}-\mathbf{B}]^k|_{z_1=\ldots=z_m=x}
	\]
	for matrix $\mathbf{B}\in\mathbb{C}^{m\times m}$,	where $\mathbf{Z} = \diag(z_1,\ldots,z_m)$.
	In \cite{ravi3}, Ravichandran and  Srivastava showed that
	\[
		\chi[\mathbf{A}_1,\ldots\mathbf{A}_k](x)= \frac{1}{(k!)^m}\prod_{i=1}^m\partial_{z_i}^{k-1}\prod_{j=1}^k\det[\mathbf{Z}-\mathbf{A}_i]|_{z_1=\ldots=z_m=x}.
	\]
	for matrices $\mathbf{A}_1,\ldots,\mathbf{A}_k\in\mathbb{C}^{m\times m}$. These two formulas correspond to the case
where $\mathbf{A}$ is a  block-diagoal matrices in Proposition \ref{prop-1}.
\end{remark}
\begin{proof}[Proof of Proposition \ref{prop-1}]
To prove the conclusion, according to (\ref{eq:AS3}), it is enough to show  that
\[
\frac{1}{k^m}\cdot D_{k,m}[x\cdot \mathbf{I}_{km}-\mathbf{A}]=
\frac{1}{(k!)^m}\cdot\prod\limits_{i=1}^{m}\partial_{z_i}^{k-1} \cdot\det[\mathbf{Z}_k-\mathbf{A}] |_{z_1=\cdots=z_m=x}.
\]
A simple calculation shows that   the following algebraic identity holds for any
polynomial $f(x_1,\ldots,x_k)$:
\begin{equation}\label{rule1}
\sum\limits_{j=1}^{k}\partial_{x_j} f(x_1,\ldots,x_k)\ |_{x_1=\cdots=x_k=x}=\partial_{x} f(x,\ldots,x).
\end{equation}
For $\rho(x_1,\ldots,x_k):=\prod_{j\in S}x_j$ with $S\subset [k]$, we have
\[
(k-1)!\cdot\sum\limits_{j=1}^{k}(\prod_{l\neq j}\partial_{x_l}) \rho\ |_{x_1=\cdots=x_k=x}=\partial_{x}^{k-1} \rho(x,\ldots,x),
\]
which implies
\begin{equation}\label{rule2}
(k-1)!\cdot\sum\limits_{j=1}^{k}(\prod_{l\neq j}\partial_{x_l}) f\ |_{x_1=\cdots=x_k=x}=\partial_{x}^{k-1} f(x,\ldots,x),
\end{equation}
where   $f(x_1,\ldots,x_k)$ is a polynomial which has degree $1$ in each $x_j$.
Now we define the polynomial
\begin{equation}\label{eq:pm0}
p_t(z_1,\ldots,z_m):=(\prod\limits_{i=1}^{t}\partial_{z_i}^{k-1} )\det[\mathbf{Z}_k-\mathbf{A}],
\end{equation}
for each $t\in[m]$. We set
\begin{equation*}
g(z_{1,1},\ldots,z_{m,1},\ldots,z_{1,k},\ldots,z_{m,k}):=\det[\mathbf{Z}-\mathbf{A}],
\end{equation*}
where
\begin{equation*}
\mathbf{Z}:=\diag(z_{1,1},\ldots,z_{m,1},\ldots,z_{1,k},\ldots,z_{m,k}).
\end{equation*}
   For each $t\in[m]$,  we use  (\ref{rule2}) to obtain that
\begin{equation}\label{eq:pkk0}
		\begin{aligned}
			p_t(z_1,\ldots,z_m)=
 ((k-1)!)^t\prod_{i=1}^{t} (\sum\limits_{j=1}^{k}(\prod_{l\neq j}\partial_{z_{i,l}}))
 g(z_{1,1},\ldots,z_{m,1},\ldots,z_{1,k},\ldots,z_{m,k})\ |_{z_{s,t}=z_{s},\ \forall s\in[k],\ \forall t\in[m]}.
		\end{aligned}
\end{equation}
A simple calculation shows that
\begin{equation}\label{eq:pk1}
\prod_{i=1}^{t} (\sum\limits_{j=1}^{k}(\prod_{l\neq j}\partial_{z_{i,l}}))=\sum\limits_{\mathcal{S}=(S_1,\ldots,S_k)\in \mathcal{P}_k(t)} \prod_{j=1}^{k}\prod_{i\in S_j}^{} \prod_{l\neq j}\partial_{z_{i,l}}.
\end{equation}
Also note that, for each $\mathcal{S}=(S_1,\ldots,S_k)\in \mathcal{P}_k(t)$, we have
\begin{equation}\label{eq:pk2}
(\prod_{j=1}^{k}\prod_{i\in S_j}^{} \prod_{l\neq j}\partial_{z_{i,l}} )\cdot g=\det[(\mathbf{Z}-\mathbf{A})(\mathcal{T}_{t,\mathcal{S}})],
\end{equation}
where $\mathcal{T}_{t,\mathcal{S}}$ is defined as
\begin{equation}\label{TtS}
\mathcal{T}_{t,\mathcal{S}}:=(S_1\cup \{t+1,t+2,\ldots,m\},\ldots,S_k\cup \{t+1,t+2,\ldots,m\})\in [m]^k.
\end{equation}
Combining (\ref{eq:pkk0}), (\ref{eq:pk1}) and (\ref{eq:pk2}), we obtain that,  for each $t\in[m]$,
\begin{equation}\label{eq:pkk}
\begin{aligned}
p_t(z_1,\ldots,z_m)&=((k-1)!)^t\sum\limits_{\mathcal{S}=(S_1,\ldots,S_k)\in \mathcal{P}_k(t)}\det[(\mathbf{Z}-\mathbf{A})(\mathcal{T}_{t,\mathcal{S}})] \ |_{z_{i,j}=z_{i},\ \forall i\in[k],\ \forall  j\in[m]}\\
&=((k-1)!)^t\sum\limits_{\mathcal{S}=(S_1,\ldots,S_k)\in \mathcal{P}_k(t)}\det[(\mathbf{Z}_k-\mathbf{A})(\mathcal{T}_{t,\mathcal{S}})].
\end{aligned}
\end{equation}
In particular, when $t=m$, we have
\begin{equation}\label{eq:pm}
	\begin{aligned}
		p_m(z_1,\ldots,z_m)=((k-1)!)^m\sum\limits_{\mathcal{S}=(S_1,\ldots,S_k)\in \mathcal{P}_k(m)}\det[(\mathbf{Z}_k-\mathbf{A})(\mathcal{S})].
	\end{aligned}
\end{equation}
Combining (\ref{eq:pm0}) and (\ref{eq:pm}), we obtain that
\begin{equation*}
\begin{aligned}
&\quad\frac{1}{(k!)^m}\cdot(\prod\limits_{i=1}^{m}\partial_{z_i}^{k-1} )\det[\mathbf{Z}_k-\mathbf{A}]\ |_{z_1=\cdots=z_m=x}\\
&=\frac{1}{k^m}\cdot \sum\limits_{\mathcal{S}=(S_1,\ldots,S_k)\in \mathcal{P}_k(m)}\det[(\mathbf{Z}_k-\mathbf{A})(\mathcal{S})]\ |_{z_1=\cdots=z_m=x}\\
&=\frac{1}{k^m}\cdot \sum\limits_{\mathcal{S}=(S_1,\ldots,S_k)\in \mathcal{P}_k(m)}\det[x\cdot \mathbf{I}_{m}-\mathbf{A}(\mathcal{S})]\\
&=\frac{1}{k^m}\cdot D_{k,m}[x\cdot \mathbf{I}_{km}-\mathbf{A}]=\psi_{k,m}[\mathbf{A}](x).
\end{aligned}
\end{equation*}
\end{proof}
The next proposition  shows  that  $\psi_{k,m}[\mathbf{A}](x)$ is real-rooted if $\mathbf{A}\in\mathbb{C}^{km\times km}$ is Hermitian.
\begin{prop}\label{prop-2}
	Let $k,m$ be two positive integers. For any Hermitian matrix $\mathbf{A}\in\mathbb{C}^{km\times km}$, the $(k,m)$-characteristic polynomial $\psi_{k,m}[\mathbf{A}](x)$ is real-rooted.
\end{prop}

\begin{proof}
Since $\mathbf{A}$ is Hermitian, by Lemma \ref{real stable}, we see that the polynomial
\begin{equation*}
\det[\mathbf{Z}_k-\mathbf{A}]\in\mathbb{R}[z_1,\ldots,z_m]
\end{equation*}
is either  real  stable or  identically zero. If it is zero, then
we are done. We assume that $\det[\mathbf{Z}_k-\mathbf{A}] \not\equiv 0$.  Lemma
\ref{real stable2} shows that the differential operator
$\prod\limits_{i=1}^{m}\partial_{z_i}^{k-1}$ and setting all variables to $x$
preserve real stability. Then, by Proposition \ref{prop-1}, we conclude that
$\psi_{k,m}[\mathbf{A}](x)$ 	is a univariate real stable polynomial, which is  real-rooted.
\end{proof}


We next present several properties about the roots of the $(k,m)$-characteristic polynomial $\psi_{k,m}[\mathbf{A}](x)$ of a Hermitian matrix $\mathbf{A}\in\mathbb{C}^{km\times km}$.

\begin{prop}\label{prop-3}
	Let $k,m$ be two positive integers. Assume that $\mathbf{A}\in\mathbb{C}^{km\times km}$ is a  Hermitian matrix. Then we have the following results:
\begin{enumerate}[{\rm (i)}]

\item The sum of the roots of $\psi_{k,m}[\mathbf{A}](x)$ equals  $\frac{\tr(\mathbf{A})}{k}$.

\item For any vector $\mathbf{v}\in\mathbb{C}^{km}$ we have 	
\begin{equation}\label{rankone}
\operatorname{maxroot}\ \psi_{k,m}[\mathbf{A}]\leq  \operatorname{maxroot}\ \psi_{k,m}[\mathbf{A+\mathbf{v}\mathbf{v}^*}].
\end{equation}	

\item If $\mathbf{A}$ is positive semidefinite, then we have
\begin{equation}\label{prop-psd}
\|\mathbf{A}\|\leq k\cdot \operatorname{maxroot}\ \psi_{k,m}[\mathbf{A}].
\end{equation}	
	
\end{enumerate}
		
	
\end{prop}

\begin{proof}

(i) We use $\alpha$  to denote the sum of the roots of $ \psi_{k,m}[\mathbf{A}](x)$. Since $ \psi_{k,m}[\mathbf{A}](x)$ is a real-rooted polynomial of degree $m$ with leading coefficient being $1$, it is known that $\alpha$ equals the negative value of the coefficient of $x^{m-1}$ in $ \psi_{k,m}[\mathbf{A}](x)$.  Note that, for each $k$-partition $\mathcal{S}$ of $[m]$, the characteristic polynomial of $\mathbf{A}( \mathcal{S})$ is of the form
\begin{equation*}
x^m-\tr(\mathbf{A}( \mathcal{S}))\cdot x^{m-1}+\ \text{lower order terms.}
\end{equation*}
According to (\ref{eq:AS3}),  we have
\begin{equation}\label{suma}
\alpha= \frac{1}{k^m}\cdot  \sum\limits_{ \mathcal{S}\in\mathcal{P}_k(m)} \tr(\mathbf{A}( \mathcal{S})).
\end{equation}
Observing that each diagonal entry of $\mathbf{A}$ appears in exactly $k^{m-1}$ distinct $k$-partitions of $[m]$, we can rewrite equation (\ref{suma}) as
\begin{equation*}
\alpha=\frac{1}{k^m}\cdot k^{m-1}\cdot \tr(\mathbf{A})=\frac{\tr(\mathbf{A})}{k},
\end{equation*}
which gives the desired result.

(ii) For each $t\in\mathbb{R}$ we set
\begin{equation*}
p_t(x):=\psi_{k,m}[\mathbf{A}+t\cdot \mathbf{v}\mathbf{v}^*](x).
\end{equation*}
According to Proposition \ref{prop-2},  $p_t(x)$ is real-rooted for each $t\in\mathbb{R}$. Define $f: \mathbb{R}\to \mathbb{R}$ as
\begin{equation*}
f(t):=\operatorname{maxroot}\ p_t(x).
\end{equation*}
Since the maximal root of a real-rooted polynomial is continuous in its coefficients, we obtain  that $f$ is a continuous function. Also note that
\begin{equation}\label{p01}
p_0(x)=\psi_{k,m}[\mathbf{A}](x),\ p_1(x)=\psi_{k,m}[\mathbf{A}+ \mathbf{v}\mathbf{v}^*](x).
\end{equation}
Hence, it is enough to show that $f(t)$ is monotone increasing in $t$ which implies  $f(0)\leq f(1)$ and hence  (\ref{rankone}).

 According to Definition \ref{def1} we have
\begin{equation*}
p_t(x)=\frac{1}{k^m}\cdot\sum\limits_{ \mathcal{S}\in \mathcal{P}_k(m) }^{}\det[x\cdot \mathbf{I}_m-\mathbf{A}( \mathcal{S})-t\cdot \mathbf{v}_{\mathcal{S}}\mathbf{v}_{\mathcal{S}}^*],
\end{equation*}
where
$\mathbf{v}_{\mathcal{S}}$  denotes the subvector of $\mathbf{v}$ by extracting the entries of $\mathbf{v}$ indexed by $ S_1\cup (m+ S_2)\cup\cdots\cup((k-1)\cdot m+ S_k)$.
Note that, for each $\mathcal{S}\in\mathcal{P}_k(m)$, $\det[x\cdot \mathbf{I}_m-\mathbf{A}( \mathcal{S})-t\cdot \mathbf{v}_{\mathcal{S}}\mathbf{v}_{\mathcal{S}}^*]$ is a polynomial in $t$ of degree at most one, which implies that we can write $p_t(x)$  in the form of
\begin{equation}\label{p10}
p_t(x)=(1-t)\cdot p_0(x)+t\cdot p_1(x).
\end{equation}

We next prove that $f$ is monotone by contradiction. For the aim of contradiction, we assume that $f$ is not monotone.  Since $f$ is continuous, there exist $t_1,s_1\in\mathbb{R}$ such that $t_1< s_1$ and $f(t_1)=f(s_1)=z_1$.  We use  (\ref{p10}) to obtain
 \begin{equation*}
(1-t_1)\cdot p_0(z_1)+t_1\cdot p_1(z_1)=(1-s_1)\cdot p_0(z_1)+s_1\cdot p_1(z_1)=0,
\end{equation*}
which implies
\begin{equation}\label{pp}
p_0(z_1)=p_1(z_1)=0.
\end{equation}
Then we substitute (\ref{pp}) into (\ref{p10}) and obtain $p_t(z_1)=0$ for any $t\in\mathbb{R}$. By the definition of $f$, this means that $f(t)\geq z_1$ for any $t\in\mathbb{R}$. Let $z_{max}:=\max\limits_{t_1\leq t\leq s_1} f(t)$.
Since $f$ is not monotone, we have $z_{max}>z_1$. Set $z_2=\frac{1}{2}\cdot(z_1+z_{max})>z_1$. By continuity, there exist $t_2,s_2\in[t_1,s_1]$ such that $t_2<s_2$ and $f(t_2)=f(s_2)=z_2$. Then the preceding argument shows that $f(t)\geq z_2$ for any $t\in\mathbb{R}$, which contradicts with $f(t_1)=f(s_1)=z_1<z_2$. Hence, $f$ is monotone.

We next show that $f$ is monotone increasing. Let $\alpha_t$ denote the sum of the roots of  $p_t(x)$. It follows from (i) that
\begin{equation*}
\alpha_t=\frac{\tr(\mathbf{A})}{k}+\frac{\mathbf{v}^*\mathbf{v}}{k}\cdot t.
\end{equation*}
Hence, when $t\to+\infty$, we have $\alpha_t\to+\infty$. Since $\alpha_t\leq m\cdot \operatorname{maxroot}\ p_t(x)$, we have
\begin{equation*}
f(t)=\operatorname{maxroot}\ p_t(x)\to +\infty
\end{equation*}
with $t\to+\infty$. Thus, $f$ is monotone increasing in $t$ which implies $f(1)\geq f(0)$. We arrive at the conclusion.

(iii)  By the spectral decomposition, we can write
\begin{equation*}
\mathbf{A}=\sum\limits_{i=1}^{km}\lambda_i(\mathbf{A}) \mathbf{v}_i\mathbf{v}_i^*,
\end{equation*}
where $\mathbf{v}_i\in\mathbb{C}^{km}$ is the unit-norm eigenvector of $\mathbf{A}$ corresponding  to the $i$-th largest eigenvalue $\lambda_i(\mathbf{A})$.
Since $\mathbf{A}$ is positive semidefinite, we have $\lambda_i(\mathbf{A})\geq0$ for each $i\in[km]$.
According to  (\ref{rankone}), we have
\begin{equation}\label{left}
\operatorname{maxroot}\ \psi_{k,m}[\lambda_1(\mathbf{A}) \mathbf{v}_1\mathbf{v}_1^*] \leq \operatorname{maxroot}\ \psi_{k,m}[\mathbf{A}].
\end{equation}
A simple calculation shows that
\[
\psi_{k,m}[\lambda_1(\mathbf{A}) \mathbf{v}_1\mathbf{v}_1^*](x)=x^{m-1}(x-\frac{1}{k}\cdot \lambda_1(\mathbf{A}) \mathbf{v}_1^*\mathbf{v}_1),
\]
which implies
\begin{equation}\label{suma1}
\operatorname{maxroot}\ \psi_{k,m}[\lambda_1(\mathbf{A}) \mathbf{v}_1\mathbf{v}_1^*]=\frac{1}{k}\cdot \lambda_1(\mathbf{A}).
\end{equation}
Since $\mathbf{A}$ is positive semidefinite, its operator norm is exactly $\lambda_1(\mathbf{A})$. Hence, combining  (\ref{left}) with (\ref{suma1}), we arrive at
\begin{equation*}
\|\mathbf{A}\|\leq  k\cdot  \operatorname{maxroot}\ \psi_{k,m}[\mathbf{A}] .
\end{equation*}

\end{proof}

\begin{remark}
In  \cite[Theorem 1.9]{ravi3},  Ravichandran and Srivastava  proved that
\begin{equation}\label{prop-psd2}
\operatorname{maxroot}\ \det[x\cdot \mathbf{I}_{km}-\mathbf{A}]\leq  k\cdot  \operatorname{maxroot}\ \psi_{k,m}[\mathbf{A}],
\end{equation}
where   $\mathbf{A}=\diag(\mathbf{A}_1,\ldots,\mathbf{A}_k)\in\mathbb{C}^{km\times km}$ is a block diagonal matrix with $\mathbf{A}_i\in\mathbb{C}^{m\times m}$  being zero-diagonal Hermitian.
Recall that  (\ref{prop-psd}) requires that $\mathbf{A}$ is  positive semidefinite.  Motivated by these results, we conjecture (\ref{prop-psd}) holds
 provided $\mathbf{A}\in \mathbb{C}^{km\times km}$ is  diagonal-non-negative  Hermitian.
\end{remark}
Inspired by \cite[Lemma 5.5]{ravi3}, we next show  a connection between the $(k,m)$-characteristic polynomial  and the mixed characteristic polynomial,
which is the main result of this section.

\begin{theorem}\label{th-mix}
Let $\mathbf{X}_1,\ldots,\mathbf{X}_m\in\mathbb{C}^{d\times d}$ be matrices of rank at most $k$. Suppose that
$\{\mathbf{u}_{i,j}\}_{i\in[m],j\in[k]}\subset\mathbb{C}^d$ and $\{\mathbf{v}_{i,j}\}_{i\in[m],j\in[k]} \subset\mathbb{C}^d$ satisfy
\begin{equation*}
\mathbf{X}_i=\sum\limits_{j=1}^{k}\mathbf{u}_{i,j}\mathbf{v}_{i,j}^*\quad \text{for all $i\in[m]$.}
\end{equation*}
We set $
\mathbf{U}_j:=[\mathbf{u}_{1,j},\ldots,\mathbf{u}_{m,j}]\in\mathbb{C}^{d\times m}, \mathbf{V}_j:=[\mathbf{v}_{1,j},\ldots,\mathbf{v}_{m,j}]\in\mathbb{C}^{d\times m}$ for all $j\in[k]$ and
\begin{equation}\label{constA}
\mathbf{A}:=\left(
\begin{matrix}
\mathbf{V}_1^*\\
\vdots\\
\mathbf{V}_k^*
\end{matrix}
\right)(\mathbf{U}_1,\ldots,\mathbf{U}_k)\in\mathbb{C}^{km\times km}.
\end{equation}
Then, we have
\begin{equation*}
\mu[\mathbf{X}_1,\ldots,\mathbf{X}_m](x)=x^{d-m}\cdot\psi_{k,m}[k\cdot\mathbf{A}](x).
\end{equation*}
\end{theorem}

\begin{proof}
For each $i\in[m]$, let $\mathbf{W}_i$ be the random rank one matrix taking values in $$\{\sqrt{k}\cdot \mathbf{u}_{i,1}\mathbf{v}_{i,1}^*,\ldots,\sqrt{k}\cdot\mathbf{u}_{i,k}\mathbf{v}_{i,k}^*\}\subset\mathbb{C}^{d\times d}$$ with equal probability. Then, we have
\begin{equation*}
\mathbb{E}\ \mathbf{W}_i=\frac{1}{k}\cdot\sum\limits_{j=1}^{k}k\cdot\mathbf{u}_{i,j}\mathbf{v}_{i,j}^*=\mathbf{X}_i.
\end{equation*}
For each subset $ S\subset[m]$ and each $j\in[k]$, we let $\mathbf{U}_{j, S}, \mathbf{V}_{j, S}\in\mathbb{C}^{d\times|S|}$ denote the submatrix of $\mathbf{U}_j$ and $\mathbf{V}_j$, respectively, obtained by extracting the columns indexed by $ S$. Now a simple calculation shows
\begin{equation*}
\begin{aligned}
\mu[\mathbf{X}_1,\ldots,\mathbf{X}_m](x)&=\mathbb{E}\ \det[x\cdot\mathbf{I}_d-\sum\limits_{i=1}^{m}\mathbf{W}_i]\\
&=\frac{1}{k^m}\sum\limits_{\mathcal{S}=(S_1,\ldots,S_k)\in\mathcal{P}_k(m)}\det[x\cdot\mathbf{I}_d-k\cdot\sum\limits_{j=1}^{k}\sum\limits_{i\in S_j}\mathbf{u}_{i,j}\mathbf{v}_{i,j}^*]\\
&=\frac{1}{k^m}\sum\limits_{\mathcal{S}=(S_1,\ldots,S_k)\in\mathcal{P}_k(m)}\det[x\cdot\mathbf{I}_d-k\cdot\sum\limits_{j=1}^{k}\mathbf{U}_{j,S_j}\mathbf{V}_{j,S_j}^*].
\end{aligned}
\end{equation*}
Note that for each $\mathcal{S}=(S_1,\ldots,S_k)\in\mathcal{P}_k(m)$ we have
\begin{equation*}
\begin{aligned}
\det[x\cdot\mathbf{I}_d-k\cdot\sum\limits_{j=1}^{k}\mathbf{U}_{j,S_j}\mathbf{V}_{j,S_j}^*]&=\det [x\cdot \mathbf{I}_d-k\cdot(\mathbf{U}_{1, S_1},\ldots,\mathbf{U}_{k, S_k})\left(
\begin{matrix}
\mathbf{V}_{1, S_1}^*\\
\vdots\\
\mathbf{V}_{k, S_k}^*
\end{matrix}
\right)]\\
&=x^{d-m}\cdot \det [x\cdot \mathbf{I}_m-k\cdot\left(
\begin{matrix}
\mathbf{V}_{1, S_1}^*\\
\vdots\\
\mathbf{V}_{k, S_k}^*
\end{matrix}
\right)(\mathbf{U}_{1, S_1},\ldots,\mathbf{U}_{k, S_k})]\\
&=x^{d-m}\cdot \det [x\cdot \mathbf{I}_m-k\cdot\mathbf{A}(\mathcal{S})].
\end{aligned}
\end{equation*}
Hence, we arrive at
\begin{equation*}
	\begin{aligned}
	\mu[\mathbf{X}_1,\ldots,\mathbf{X}_m](x)&=x^{d-m}\cdot\frac{1}{k^m}\sum\limits_{\mathcal{S}=(S_1,\ldots,S_k)\in\mathcal{P}_k(m)}  \det [x\cdot \mathbf{I}_m-k\cdot\mathbf{A}(\mathcal{S})]
	\\&=x^{d-m}\cdot\psi_{k,m}[k\cdot\mathbf{A}](x).
	\end{aligned}
\end{equation*}

\end{proof}

\begin{remark}
Note that  $\operatorname{rank}(\mathbf{A})\leq d$ where $\mathbf{A}\in\mathbb{C}^{km\times km}$  is presented in  (\ref{constA}).
We next  assume that $\mathbf{X}_1,\ldots,\mathbf{X}_m\in\mathbb{C}^{d\times d}$ are positive semidefinite. Combining  Proposition \ref{prop-2} and Theorem \ref{th-mix}, we obtain that $\mu[\mathbf{X}_1,\ldots,\mathbf{X}_m]$ is real-rooted, which conincides with the well-known result in \cite[Corollary 4.4]{inter2}.
Note that
\[\psi_{k,m}[k\cdot\mathbf{A}](x)=k^{m}\cdot\psi_{k,m}[\mathbf{A}](\frac{x}{k}).
\]
It immediately follows from Theorem \ref{th-mix} that
\begin{equation}\label{mix-equal}
\operatorname{maxroot}\ \mu[\mathbf{X}_1,\ldots,\mathbf{X}_m]=\operatorname{maxroot}\ \psi_{k,m}[k\cdot\mathbf{A}]=k\cdot \operatorname{maxroot}\ \psi_{k,m}[\mathbf{A}].
\end{equation}
Also note that
 \begin{equation}\label{eq:XAdeng}
 \bigg\|\sum_{i=1}^m\mathbf{X}_i\bigg\|\,\,=\,\, \| \mathbf{A}\|.
 \end{equation}
  Combining  Proposition \ref{prop-3} (iii) , (\ref{mix-equal}) and (\ref{eq:XAdeng}), we obtain that
\begin{equation}\label{}
\bigg\|\sum\limits_{i=1}^{m}\mathbf{X}_i\bigg\|\leq    \operatorname{maxroot}\ \mu[\mathbf{X}_1,\ldots,\mathbf{X}_m],
\end{equation}
which is the same with the result in Proposition \ref{mix-th1}.
\end{remark}

\section{Proof of Theorem \ref{maxroot-mu}}\label{s:barrier}

The aim of this section is to prove Theorem \ref{maxroot-mu}, which shows an upper bound for the maximum root of  $\mu[\mathbf{X}_1,\ldots,\mathbf{X}_m]$. To do that, we first present an upper bound of the maximum root of $\psi_{k,m}[\mathbf{A}](x)$, and then we   employ the connection between $\mu[\mathbf{X}_1,\ldots,\mathbf{X}_m]$ and  $\psi_{k,m}[\mathbf{A}](x)$  (see equation (\ref{mix-equal})) to prove Theorem \ref{maxroot-mu}.


\begin{theorem}\label{maxroot-A}
Let $k\geq2$ and $m\geq 1$ be two positive integers. Assume that $\mathbf{A}\in\mathbb{C}^{km\times km}$ is a positive semidefinite Hermitian matrix such that $\mathbf{0}\preceq\mathbf{A}\preceq\mathbf{I}_{km}$. Set
\[
\varepsilon:=\max\limits_{1\leq i\leq m}\sum\limits_{j=1}^{k}\mathbf{A}(i+(j-1)m,i+(j-1)m).
\]
If $\varepsilon\leq\frac{(k-1)^2}{k}$, then we have
\begin{equation}\label{maxroot2}
\operatorname{maxroot}\ \psi_{k,m}[\mathbf{A}]\leq \frac{1}{k}\cdot \bigg(\sqrt{1-\frac{\varepsilon}{k-1}}+\sqrt{\varepsilon}\bigg)^2.
\end{equation}

\end{theorem}
\begin{remark}
Theorem \ref{maxroot-A}  generalizes Ravichandran and  Leake's result in \cite[Theorem 9]{ravi2} which presents  an upper bound for the largest root of the $k$-characteristic polynomial  $\chi_k[\mathbf{B}]=k^m\psi_{k,m}[{\rm diag}(\underbrace{\mathbf{B},\ldots,\mathbf{B}}_k)]$ for matrix $\mathbf{B}\in\mathbb{C}^{m\times m}$. Particularly, they showed that
	\begin{equation}\label{maxroot3}
		\operatorname{maxroot}\  \chi_k[\mathbf{B}]\leq \frac{1}{k}\cdot \bigg(\sqrt{1-\frac{k\alpha}{k-1}}+\sqrt{k\alpha}\bigg)^2
	\end{equation}
provided 	 $\alpha : = \max_{1\leq i\leq m} \mathbf{B}(i,i)\leq \frac{(k-1)^2}{k^2}$.
\end{remark}


We next give a proof of Theorem \ref{maxroot-mu}. The proof of Theorem \ref{maxroot-A} is postponed to the end of this section.

\begin{proof}[Proof of Theorem \ref{maxroot-mu}]
	Recall that $\mathbf{X}_1,\ldots,\mathbf{X}_m\in\mathbb{C}^{d\times d}$ are positive semidefinite Hermitian matrices of rank at most $k$ such that $\sum\limits_{i=1}^{m}\mathbf{X}_i\preceq\mathbf{I}_d$. Suppose that
	$\{\mathbf{u}_{i,j}\}_{i\in[m],j\in[k]}$ are the vectors in $\mathbb{C}^d$ such that
	\begin{equation*}
		\mathbf{X}_i=\sum\limits_{j=1}^{k}\mathbf{u}_{i,j}\mathbf{u}_{i,j}^*\quad \text{for all $i\in[m]$.}
	\end{equation*}
	We set $\mathbf{U}:=[\mathbf{u}_{1,1},\ldots,\mathbf{u}_{m,1},\ldots,\mathbf{u}_{1,k},\ldots,\mathbf{u}_{m,k}]\in\mathbb{C}^{d\times km}$ and set $\mathbf{A}=\mathbf{U}^*\mathbf{U}$.
	Note that
	\begin{equation*}
		\mathbf{0}\preceq\sum\limits_{i=1}^{m}\mathbf{X}_i=\sum\limits_{i=1}^{m}\sum\limits_{j=1}^{k}\mathbf{u}_{i,j}\mathbf{u}_{i,j}^*=\mathbf{U}\mathbf{U}^*\preceq\mathbf{I}_d.
	\end{equation*}
	Since $\mathbf{U}\mathbf{U}^*$ and $\mathbf{U}^*\mathbf{U}$ have the same nonzero eigenvalues, we obtain \[
	\mathbf{0}\preceq\mathbf{A}\preceq\mathbf{I}_{km}.
	\]
	Moreover, for each $i\in[m]$ we have
	\begin{equation*}
		\tr(\mathbf{X}_i)=\tr(\sum\limits_{j=1}^{k}\mathbf{u}_{i,j}\mathbf{u}_{i,j}^*)=\sum\limits_{j=1}^{k}\mathbf{A}(i+(j-1)m,i+(j-1)m) \leq \varepsilon.
	\end{equation*}
	Since $\varepsilon\leq (k-1)^2/k$, Theorem \ref{maxroot-A} gives
	\begin{equation*}
		\operatorname{maxroot}\ \psi_{k,m}[\mathbf{A}]\leq \frac{1}{k}\cdot \bigg(\sqrt{1-\frac{\varepsilon}{k-1}}+\sqrt{\varepsilon}\bigg)^2,
	\end{equation*}
which implies
	\begin{equation*}
		\operatorname{maxroot}\ \mu[\mathbf{X}_1,\ldots,\mathbf{X}_m]\leq \bigg(\sqrt{1-\frac{\varepsilon}{k-1}}+\sqrt{\varepsilon}\bigg)^2.
	\end{equation*}
	Here, we use Theorem \ref{th-mix}, i.e.,
	\begin{equation*}
		\operatorname{maxroot}\ \mu[\mathbf{X}_1,\ldots,\mathbf{X}_m]=k\cdot  \operatorname{maxroot}\ \psi_{k,m}[\mathbf{A}].
	\end{equation*}
\end{proof}

In the remainder of this section, we aim to prove Theorem \ref{maxroot-A} by adapting the method  of  multivariate barrier function. We first introduce the barrier function of a real stable polynomial (see also \cite{inter2}).


\begin{definition}\label{above}
Let $p(z_1,\ldots,z_m)\in\mathbb{R}[z_1,\ldots,z_m]$ be a multivariate polynomial. We say that  a  point $\mathbf{z}\in\mathbb{R}^m$ is \emph{above the roots of} $p(z_1,\ldots,z_m)$ if $p(\mathbf{z}+\mathbf{t})\neq  0$ for all $\mathbf{t}\in\mathbb{R}_{\geq 0}^m$.
We use $\mathbf{Ab}_p$ to denote the set of points that are above  the roots of $p$.
\end{definition}

\begin{definition}
Let $p\in\mathbb{R}[z_1,\ldots,z_m]$ be a real stable polynomial. The \emph{barrier function} of $p$ at a point $\mathbf{z}\in \mathbf{Ab}_p$  in the direction $i$  is defined as
\begin{equation}\label{eq:phi}
\Phi_{p}^i(\mathbf{z}):=\frac{\partial_{z_i}p}{p}(\mathbf{z}).
\end{equation}
\end{definition}

 Here we briefly introduce our proof of Theorem \ref{maxroot-A}.
According to  Proposition \ref{prop-1},  we have
\[
\psi_{k,m}[\mathbf{A}](x)=\frac{1}{(k!)^m}\cdot\prod\limits_{i=1}^{m}\partial_{z_i}^{k-1} \cdot\det[\mathbf{Z}_k-\mathbf{A}] |_{z_1=\cdots=z_m=x},
\]
where  $\mathbf{A}\in\mathbb{C}^{km\times km}$.
For Hermitian matrix $\mathbf{A}\in\mathbb{C}^{km\times km}$, we consider the real stable polynomial
\begin{equation}\label{eq:p}
p_0(z_1,\cdots,z_m):=\det[\mathbf{Z}_k-\mathbf{A}],
\end{equation}
where
\begin{equation}\label{eq:Zr}
\mathbf{Z}_k=\diag(\underbrace{\mathbf{z},\ldots,\mathbf{z}}_{k}),\quad \mathbf{z}=(z_1,\ldots,z_m).
\end{equation}
 We iteratively define the polynomials
\begin{equation}\label{eq:pt}
 p_t(z_1,\ldots,z_m):=\partial_{z_t}^{k-1}p_{t-1}(z_1,\ldots,z_m),\quad t=1,\ldots,m.
\end{equation}
We start with an initial point $\mathbf{b}_0:=a\mathbf{1}\in\mathbb{R}^m$ with $a>1$ that is above  the roots of $p_0$.
For each $t\in[m]$, we aim to find a positive number $\delta_t$ such that  each entries of $\mathbf{b}_m=\mathbf{b}_0-\sum_{t=1}^m\delta_t\mathbf{e}_t\in \mathbf{ Ab}_{p_0}$ less than or equal to
 $\frac{1}{k}(\sqrt{1-\frac{\epsilon}{k-1}}+\sqrt{\epsilon})^2$, which implies the result in Theorem \ref{maxroot-A}.
%

%

	We need  the following result which depicts the behavior of barrier functions of polynomials affected by taking derivatives.
	
	\begin{prop}\label{lem:barfun1}\cite[Proposition 9, Proposition 10]{ravi2}
		Let $j\in[m]$ be any integer. Assume that $p(z_1,\ldots,z_m)$ is a real stable polynomial of degree $k$ in $z_j$. Let $\mathbf{a}=(a_1,\ldots,a_m)\in \mathbf{Ab}_p$ and let $\lambda_k$ be the smallest root of the univariate polynomial $q(z):=p(a_1,\cdots,a_{j-1},z,a_{j+1},\cdots,a_m)$.
		Let
		\begin{equation}\label{lemma:delta}
			\delta = \frac{(k-1)^2}{k}\bigg(\frac{1}{\Phi_p^j(\mathbf{a})-\frac{1}{a_j-\lambda_k}}\bigg).
		\end{equation}
		Then, we have $\mathbf{a}-\delta\cdot \mathbf{e}_j\in \mathbf{Ab}_{\partial_{z_j}^{k-1}p}$ and
	\begin{equation*}
		\Phi_{\partial_{z_j}^{k-1}p}^i(\mathbf{a}-\delta\cdot \mathbf{e}_j)\leq \Phi_p^i(\mathbf{a}).
	\end{equation*}
	\end{prop}

	We introduce a well-known result of the zeros of the real stable polynomials (see also \cite[Proposition 8]{ravi2}).
\begin{lemma}\label{lem:staroot}{\cite[Lemma 17]{tao2013}}
Let $p(z_1,z_2)$ be a real stable polynomial of degree $k$ in $z_2$. For each $(a,b)\in \mathbf{Ab}_p$, denote the roots of the univariate polynomial $q_a(z):= p(a,z)$ by $\lambda_k(a)\leq\cdots\leq \lambda_1(a)$. Then for each $j\in[k]$, the map
$a\mapsto\lambda_j(a)$ defined on $\{a\in \mathbb{R}:(a,b)\in \mathbf{Ab}_p \,\,\text{for some}\,\,b\in\mathbb{R}\}$ is non-increasing.
\end{lemma}

The following lemma is useful for our analysis.
\begin{lemma}\label{lem:posroot}
Let $\mathbf{A}\in \mathbb{C}^{m\times m}$ be a positive semidefinite Hermitian matrix. Set
\begin{equation}\label{eq:det}
p(z_1,\ldots,z_m) := \det[\operatorname{diag}(z_1,\ldots,z_m)-\mathbf{A}] .
\end{equation}
 Assume that $\mathbf{a} = (a_1,\ldots,a_m)\in\mathbf{Ab}_p$. Then, for each $t\in[m]$, the smallest root of the univariate polynomial $q(z): = p(\underbrace{z,\ldots,z}_{t},a_{t+1},\ldots,a_m )$ is non-negative.
\end{lemma}	
\begin{proof}
We can  write  $\mathbf{A}$ in the form of
$\begin{pmatrix}\mathbf{P}&\mathbf{Q}\\ \mathbf{Q}^*&\mathbf{S}\end{pmatrix}$,
 where $\mathbf{P}\in\mathbb{C}^{t\times t}, \mathbf{S}\in\mathbb{C}^{(m-t)\times (m-t)}$ are both positive semidefinite Hermitian and $\mathbf{Q}\in\mathbb{C}^{t\times(m-t)}$. Since $\mathbf{a} = (a_1,\ldots,a_m)\in\mathbf{Ab}_p$,  we obtain that $\diag(a_1,\ldots,a_m)-\mathbf{A}$ is positive definite. Set $\mathbf{D}_t:=\diag(a_{t+1},\ldots,a_m)$. Noting that $\mathbf{D}_t-\mathbf{S}\in\mathbb{C}^{(m-t)\times (m-t)}$ is a principal submatrix of $\diag(a_1,\ldots,a_m)-\mathbf{A}$,  we obtain that $\mathbf{D}_t-\mathbf{S}$ is  positive definite. We use the Schur complement to obtain that
\begin{equation*}
\begin{aligned}
q(z)&=\det\begin{pmatrix}
z\cdot\mathbf{I}_t-\mathbf{P}&-\mathbf{Q}\\ -\mathbf{Q}^*&\mathbf{D}_t-\mathbf{S}
\end{pmatrix}\\
&=\det(\mathbf{D}_t-\mathbf{S})\det\left(z\cdot\mathbf{I}_t-\mathbf{P}-\mathbf{Q}(\mathbf{D}_t-\mathbf{S})^{-1}\mathbf{Q}^*\right).
\end{aligned}
\end{equation*}
Hence, the roots of  $q(z)$ are the eigenvalues of the positive semidefinite matrix $\mathbf{P}+\mathbf{Q}(\mathbf{D}_t-\mathbf{S})^{-1}\mathbf{Q}^*\in\mathbb{C}^{t\times t}$, which implies that all roots of $q(z)$ are non-negative.
\end{proof}

We next extend Lemma \ref{lem:posroot} to the polynomials that we will use in our proof of Theorem \ref{maxroot-A}.

\begin{prop}\label{prop:miniroot}
Let $k,m$ be two positive integers. Let $\mathbf{A}\in\mathbb{C}^{km\times km}$ be a positive semidefinite Hermitian matrix. 
Let $p_0(z_1,\ldots,z_m)$ be defined as in (\ref{eq:p}).
 Assume that $\mathbf{a} = (a_1,\ldots,a_m)\in\mathbf{Ab}_{p_0}$.  For each $t\in[m]$, let $p_t(z_1,\ldots,z_m)$ be defined as in (\ref{eq:pt}).
  Set
$$
 q_t(z):=p_{t-1}(a_1,\ldots,a_{t-1},z,a_{t+1},\ldots,a_m).
$$
Then, for each $t\in[m]$, the smallest root of $q_t(z)$ is non-negative.
\end{prop}
\begin{proof}
By equation (\ref{eq:pkk}), we can write the univariate polynomial $q_t(z)$ in the form of
\begin{equation*}
q_t(z)=((k-1)!)^{(t-1)}\sum\limits_{\mathcal{S}=(S_1,\ldots,S_k)\in \mathcal{P}_k(t-1)} g_{\mathcal{S}}(a_1,\ldots,a_{t-1},z, a_{t+1},\ldots,a_n),
\end{equation*}
where
\[
g_\mathcal{S}(z_1,\ldots,z_m):=\det[(\mathbf{Z}_k-\mathbf{A})(\mathcal{T}_{t-1,\mathcal{S}})],
\]
and
 $\mathcal{T}_{t-1,\mathcal{S}}$ is defined  in (\ref{TtS}).
Since $\mathbf{a}\in\mathbf{Ab}_{p_0}$, we have
\[
	\operatorname{diag}(\mathbf{a},\ldots,\mathbf{a})-\mathbf{A}\succ\mathbf{0} ,
\]
which implies that the principal matrix
\[
	(\operatorname{diag}(\mathbf{a},\ldots,\mathbf{a})-\mathbf{A})(\mathcal{T}_{t-1,\mathcal{S}})\succ \mathbf{0}.
\]
Then we obtain that $\mathbf{a}\in \mathbf{Ab}_{g_{\mathcal{S}}}$. 
  Lemma \ref{lem:posroot} shows  that the smallest root of the polynomial
  \[
  g_{\mathcal{S}}(a_1,\ldots,a_{t-1},z, a_{t+1},\ldots,a_n)\in\mathbb{R}[z]
  \]
  is non-negative. Since 
$q_{t}(z)$ is a multiple of the sum of monic polynomials whose smallest root is non-negative, it follows that the smallest root of $q_{t}(z)$ is non-negative.
\end{proof}

We need the following lemma to provide an estimate on the value of $\delta_t$.

\begin{lemma}\label{lem:inibar}
Let $k,m$ be two positive integers. Let $\mathbf{A}\in\mathbb{C}^{km\times km}$ be a positive semidefinite Hermitian matrix satisfying $\mathbf{0}\preceq \mathbf{A}\preceq \mathbf{I}_{km}$. Set
\[
\varepsilon := \max_{1\leq i\leq m}\sum_{j=1}^k \mathbf{A}(i+(j-1)m,i+(j-1)m).
\]
Let $p_0(z_1,\ldots,z_m)$ be defined as in (\ref{eq:p}).
 Then for any $a>1$ we have
\[
\Phi_{p_0}^i(a\mathbf{1})\leq \frac{\varepsilon}{a-1}+\frac{k-\varepsilon}{a}\quad \text{for all $i\in[m]$}.
\]
\end{lemma}
\begin{proof}
Define the polynomial
\begin{equation*}
g(z_1,\ldots,z_{km}) :=\det[\mathbf{Z}-\mathbf{A}],
\end{equation*}
where
\begin{equation*}
	\mathbf{Z}:=\diag(z_1,\ldots,z_{km}).
\end{equation*}
By equation (\ref{rule1}), for each $i\in[m]$, we have
\begin{equation}\label{eq:phip1}
\Phi_{p_0}^i(a\mathbf{1}) = \frac{\partial_{z_i}p_0}{p_0}(a\mathbf{1}) = \sum_{j=1}^k\frac{\partial_{z_{i+(j-1)m}}g}{g}\bigg|_{\mathbf{Z} = a\mathbf{I}_{km}}.
\end{equation}
For each $i\in[m], j\in[k]$, we have $$\partial_{z_{i+(j-1)m}}\det[\mathbf{Z}-\mathbf{A}]=\det[(\mathbf{Z}-\mathbf{A})_{i+(j-1)m}],$$
where $(\mathbf{Z}-\mathbf{A})_{i+(j-1)m}$ denotes the principle submatrix of $\mathbf{Z}-\mathbf{A}$ whose the $(i+(j-1)m)$-th row and $(i+(j-1)m)$-th column  are deleted. By the expression of the inverse of a matrix, we obtain that
\begin{equation}\label{eq:inv}
(\mathbf{Z}-\mathbf{A})^{-1}= \frac{1}{\det[\mathbf{Z}-\mathbf{A}]}\cdot\operatorname{adj}(\mathbf{Z}-\mathbf{A}).
\end{equation}
 Taking the $(i+(j-1)m)$-th diagonal element of each side of (\ref{eq:inv}), we have
\[
\mathbf{e}_{i+(j-1)m}^*(\mathbf{Z}-\mathbf{A})^{-1}\mathbf{e}_{i+(j-1)m}  = \frac{1}{\det[\mathbf{Z}-\mathbf{A}]}\cdot\det[(\mathbf{Z}-\mathbf{A})_{i+(j-1)m}]
\]
Then for each $i\in[m], j\in[k]$, we have
\begin{equation}\label{eq:phip2}
\frac{\partial_{z_{i+(j-1)m}}g}{g}=\frac{\partial_{z_{i+(j-1)m}}\det[\mathbf{Z}-\mathbf{A}]}{\det[\mathbf{Z}-\mathbf{A}]}=\mathbf{e}_{i+(j-1)m}^*(\mathbf{Z}-\mathbf{A})^{-1}\mathbf{e}_{i+(j-1)m}.
\end{equation}
Combining (\ref{eq:phip1}) and (\ref{eq:phip2}), we obtain that
\begin{equation*}
\Phi_{p_0}^i(a\mathbf{1}) =\sum\limits_{j=1}^{k}\mathbf{e}_{i+(j-1)m}^*(a\mathbf{I}_{km}-\mathbf{A})^{-1}\mathbf{e}_{i+(j-1)m}.
\end{equation*}
Suppose that the spectral decomposition of $\mathbf{A}$ is
\begin{equation}\label{eq:ed}
\mathbf{A} = \mathbf{U}\mathbf{D}\mathbf{U}^*,
\end{equation} where $\mathbf{U}\in\mathbb{C}^{km\times km}$ is a unitary matrix and $\mathbf{D}=\diag(\lambda_1,\ldots,\lambda_{km})$ is a diagonal matrix of eigenvalues of $\mathbf{A}$. It follows that $(a\mathbf{I}_{km}-\mathbf{A})^{-1}=\mathbf{U}(a\mathbf{I}_{km}-\mathbf{D})^{-1}\mathbf{U}^*$. Thus we have
\begin{equation}\label{write1}
	\begin{aligned}
		\Phi_{p_0}^i(a\mathbf{1}) &=\sum\limits_{j=1}^{k}\mathbf{e}_{i+(j-1)m}^*\mathbf{U}(a\mathbf{I}_{km}-\mathbf{D})^{-1}\mathbf{U}^*\mathbf{e}_{i+(j-1)m}\\
		&=\sum\limits_{j=1}^{k}\sum\limits_{l=1}^{km}\frac{|\mathbf{U}(i+(j-1)m,l)|^2}{a-\lambda_l}.
	\end{aligned}
\end{equation}
Set $c_l:=\sum\limits_{j=1}^{k}|\mathbf{U}(i+(j-1)m,l)|^2$ for each $l\in[km]$.
 Then we have
\begin{equation}\label{eq:zuhe1}
	\begin{aligned}
		\Phi_{p_0}^i(a\mathbf{1}) =\sum\limits_{l=1}^{km}\frac{c_l}{a-\lambda_l}
	&\leq \sum\limits_{l=1}^{km}\bigg(\frac{c_l\cdot \lambda_l}{a-1}+\frac{c_l\cdot(1-\lambda_l)}{a}\bigg)\\
	&=\bigg(\frac{1}{a-1}-\frac{1}{a}\bigg)\cdot \sum\limits_{l=1}^{km}\lambda_lc_l +\frac{1}{a}\cdot \sum\limits_{l=1}^{km}c_l.
	\end{aligned}
\end{equation}
Here, we use the following inequality:
\begin{equation*}
\frac{1}{a-\lambda_l}\leq \frac{\lambda_l}{a-1}+\frac{1-\lambda_l}{a},
\end{equation*}
where $ a>1$ and $\lambda_l\in [0,1]$.
Note  that 
\begin{equation}\label{eq:zuhe2}
\sum_{l=1}^{km}c_l = \sum_{j=1}^k\sum_{l=1}^{km}|\mathbf{U}(i+(j-1)m,l)|^2=\sum_{j=1}^k 1=k,
\end{equation}
since each column of $\mathbf{U}$ is a unit vector. We  use  (\ref{eq:ed}) to obtain that
\[
\mathbf{A}(i+(j-1)m,i+(j-1)m)=\sum_{l=1}^{km}\lambda_l|\mathbf{U}(i+(j-1)m,l)|^2.
\]
Then we have
\begin{equation}\label{eq:zuhe3}
\begin{aligned}
\sum_{l=1}^{km}\lambda_lc_l  &= \sum_{j=1}^k\sum_{l=1}^{km}\lambda_l|\mathbf{U}(i+(j-1)m,l)|^2\\
&=\sum_{j=1}^k \mathbf{A}(i+(j-1)m,i+(j-1)m)\\
&\leq \varepsilon.
\end{aligned}
\end{equation}
Thus, combining (\ref{eq:zuhe1}), (\ref{eq:zuhe2}) and (\ref{eq:zuhe3}), we obtain that
\[
\Phi_{p_0}^i(a\mathbf{1}) \leq \varepsilon\cdot\bigg(\frac{1}{a-1}-\frac{1}{a}\bigg)+\frac{k}{a}= \frac{\varepsilon}{a-1}+\frac{k-\varepsilon}{a}.
\]
\end{proof}

Now we have all the materials to prove Theorem \ref{maxroot-A}.
\begin{proof}[Proof of Theorem \ref{maxroot-A}]
 Suppose that $a>1$ is a constant which is chosen later. Set $\mathbf{b}_0 = a\mathbf{1}\in \R^m$.
 Let $p_0(z_1,\ldots,z_m)$ be defined as in (\ref{eq:p}).
According to  $\mathbf{A}\preceq \mathbf{I}_{km}$, we obtain that $\mathbf{b}_0$ is above the roots of $p_0$ .
Recall that
\begin{equation*}
p_t=\partial_{z_t}^{k-1}p_{t-1},\quad t=1,\ldots,m.
\end{equation*}
A simple observation is  that
\[
\psi_{k,m}[\mathbf{A}](x)=\frac{1}{(k!)^m}\cdot p_m(z_1,\ldots,z_m)|_{z_1=\cdots=z_m=x}.
\]
Hence,  to prove the conclusion, it is enough to show that $\frac{1}{k}\big(\sqrt{1-\frac{\varepsilon}{k-1}}+\sqrt{\varepsilon}\big)^2\cdot \mathbf{1}$
is above the roots of $p_m$.
For each $t\in[m]$, set
\[
\delta_t := \frac{(k-1)^2}{k}\bigg(\frac{1}{\Phi_{p_{t-1}}^t(\mathbf{b}_{t-1})-\frac{1}{a-\lambda_k^{(t)}}}\bigg),
\]
where
$$
\mathbf{b}_t :=\mathbf{b}_0-\sum_{j=1}^t\delta_j{\mathbf e}_j=(a-\delta_1,\ldots,a-\delta_{t},a,\ldots,a).
$$
Let $\lambda_k^{(t)}$ be the smallest root of the univariate polynomial
\begin{equation}\label{eq:unipoly}
q_{t}(z):= p_{t-1}(a-\delta_1,\ldots,a-\delta_{t-1},z,a,\ldots,a).
\end{equation}
By Proposition \ref{lem:barfun1} , we have $\mathbf{b}_t\in \mathbf{Ab}_{p_t}$ and
\begin{equation}\label{mono}
\Phi_{p_t}^i(\mathbf{b}_t)\leq \Phi_{p_{t-1}}^i(\mathbf{b}_{t-1}) \quad \text{for all $t\in[m]$ and $i\in[m]$}.
\end{equation}
we claim that
\begin{equation}\label{eq:budeng1}
\max_{t\in [m]}\inf_{a>1}(a-\delta_t)\leq \frac{1}{k}\bigg(\sqrt{1-\frac{\varepsilon}{k-1}}+\sqrt{\varepsilon}\bigg)^2.
\end{equation}
Noting that $\mathbf{b}_m=a\mathbf{1}-\sum\limits_{t=1}^{m}\delta_t\cdot\mathbf{e}_t$  is above the roots of $p_m$,  we can use (\ref{eq:budeng1})
 to obtain the conclusion.

 It remains to prove (\ref{eq:budeng1}).
We first show that $\lambda_k^{(t)}\geq 0$ for each $t\in[m]$. For each $t\in[m]$ we define the univariate polynomial
\begin{equation}
g_{t}(z):= p_{t-1}\left(\underbrace{a,\ldots,a}_{t-1},z,a,\ldots,a\right).
\end{equation}
Since $\mathbf{a}$ and $\mathbf{b}_{t-1}=\mathbf{a}-\sum\limits_{i=1}^{t-1}\delta_i\cdot \mathbf{e}_i\in \mathbf{Ab}_{p_{t-1}}$,  we obtain that
\begin{equation}\label{eq:noneg}
\lambda_k^{(t)}=\text{minroot}\ q_{t} \geq \text{minroot}\ g_{t}\geq 0,
\end{equation}
where the first inequality follows from Lemma \ref{lem:staroot} and the second inequality follows from
 Proposition \ref{prop:miniroot}.
Then, combining (\ref{eq:noneg}), (\ref{mono}) and Lemma \ref{lem:inibar}, for each $t\in[m]$, we have
\begin{equation*}
\begin{aligned}
\delta_t &\geq\frac{(k-1)^2}{k}\bigg(\frac{1}{\Phi_{p_{t-1}}^t(\mathbf{b}_{t-1})-\frac{1}{a}}\bigg)\\
&\geq \frac{(k-1)^2}{k}\bigg(\frac{1}{\Phi_{p_0}^t(\mathbf{b}_{0})-\frac{1}{a}}\bigg)\\
&\geq \frac{(k-1)^2}{k}\bigg(\frac{1}{\frac{\varepsilon}{a-1}+\frac{k-\varepsilon}{a}-\frac{1}{a}}\bigg).
\end{aligned}
\end{equation*}
So, for each $t\in [m]$, we have
\begin{equation}\label{eq:budeng2}
\inf\limits_{a>1} (a-\delta_t)\,\,\leq\,\, \inf\limits_{a>1} \bigg(a-\frac{(k-1)^2}{k}\bigg(\frac{1}{\frac{\varepsilon}{a-1}+\frac{k-\varepsilon}{a}-\frac{1}{a}}\bigg)\bigg).
\end{equation}
Set $\alpha:=a-(1-\frac{\varepsilon}{k-1})>0$. A simple calculation shows that
\begin{equation*}
\begin{aligned}
a-\frac{(k-1)^2}{k}\bigg(\frac{1}{\frac{\varepsilon}{a-1}+\frac{k-\varepsilon}{a}-\frac{1}{a}}\bigg)& = \frac{1}{k}\cdot  \bigg(\alpha+\frac{(1-\frac{\varepsilon}{k-1})\varepsilon}{\alpha}+(1-\frac{\varepsilon}{k-1})+\varepsilon\bigg)\\
& \geq \frac{1}{k}\cdot  \bigg(2\sqrt{(1-\frac{\varepsilon}{k-1})\varepsilon}+(1-\frac{\varepsilon}{k-1})+\varepsilon\bigg)\\
& = \frac{1}{k}\cdot\bigg(\sqrt{1-\frac{\varepsilon}{k-1}}+\sqrt{\varepsilon}\bigg)^2,
\end{aligned}
\end{equation*}
where the equality holds if and only if $\alpha=\sqrt{(1-\frac{\varepsilon}{k-1})\varepsilon}$, i.e.,
\[
a=a_0:=\sqrt{(1-\frac{\varepsilon}{k-1})\varepsilon}+(1-\frac{\varepsilon}{k-1}).
\]
This implies that if $a_0\geq 1$, i.e., $\varepsilon\leq (k-1)^2/k$, then
\begin{equation}\label{eq:budeng3}
 \inf\limits_{a>1} \bigg(a-\frac{(k-1)^2}{k}\bigg(\frac{1}{\frac{\varepsilon}{a-1}+\frac{k-\varepsilon}{a}-\frac{1}{a}}\bigg)\bigg)=\frac{1}{k}\cdot\bigg(\sqrt{1-\frac{\varepsilon}{k-1}}+\sqrt{\varepsilon}\bigg)^2.
\end{equation}
Combing (\ref{eq:budeng2}) and (\ref{eq:budeng3}), we arrive at (\ref{eq:budeng1}).
\end{proof}

\section{Proof of Theorem \ref{mth0}}\label{s:proof}
	Following  \cite{coh2016} and \cite[Theorem 6.1]{Branden2}, we now prove Theorem \ref{mth0} by employing Theorem \ref{maxroot-mu}. 
\begin{proof}[Proof of Theorem \ref{mth0}]
	For each $i\in[m]$, let $W_i:=\{\mathbf{W}_{i,1},\ldots,\mathbf{W}_{i,l_i}\}$ be  the support of $\mathbf{W}_i$. By Lemma \ref{mix-th2}, we see that the polynomials
	\begin{equation*}
		\mu[\mathbf{W}_{1,j_1},\ldots,\mathbf{W}_{m,j_m}](x),\, j_i\in[l_i], i=1,\ldots,m
	\end{equation*}
	form an interlacing family. Then Lemma \ref{interlacing2} implies that there exists $j_1\in[l_1],\ldots,j_m\in[l_m]$  such that
	\begin{equation}\label{eq:up1}
		\operatorname{maxroot}\ \mu[\mathbf{W}_{1,j_1},\ldots,\mathbf{W}_{m,j_m}]\leq \operatorname{maxroot}\ \mathbb{E}\  \mu[\mathbf{W}_1,\ldots,\mathbf{W}_m].
	\end{equation}
	Combining  (\ref{mix-th3}) and (\ref{eq:up1}), we obtain that
	\begin{equation}\label{eq:fin}
		\operatorname{maxroot}\ \mu[\mathbf{W}_{1,j_1},\ldots,\mathbf{W}_{m,j_m}]\leq \operatorname{maxroot} \ \mu[\mathbb{E}\mathbf{W}_1,\ldots,\mathbb{E}\mathbf{W}_m].
	\end{equation}
	Since
	\[
	\operatorname{tr}(\mathbb{E}\mathbf{W}_i)\leq \varepsilon,\quad\,\operatorname{rank}(\mathbb{E}\mathbf{W}_i)\leq k\quad \text{for all $i\in[m]$},
	\]
	and $\sum_{i=1}^m \mathbb{E}\mathbf{W}_i = \mathbf{I}_d$,
	Theorem \ref{maxroot-mu} gives
	\begin{equation*}
		\operatorname{maxroot} \ \mu[\mathbb{E}\mathbf{W}_1,\ldots,\mathbb{E}\mathbf{W}_m]\leq \bigg(\sqrt{1-\frac{\varepsilon}{k-1}}+\sqrt{\varepsilon}\bigg)^2.
	\end{equation*}
	Finally, combining Lemma \ref{mix-th1} and (\ref{eq:fin}) we arrive at
	\begin{equation*}
		\bigg\|\sum_{i=1}^m\mathbf{W}_{i,j_i}\bigg\|\leq	\operatorname{maxroot}\ \mu[\mathbf{W}_{1,j_1},\ldots,\mathbf{W}_{m,j_m}]\leq \bigg(\sqrt{1-\frac{\varepsilon}{k-1}}+\sqrt{\varepsilon}\bigg)^2.
	\end{equation*}
	This implies the desired conclusion.
	
\end{proof}

\end{document}